\documentclass{amsart}

\usepackage{amsmath} 
\usepackage{amssymb}
\usepackage{mathrsfs}

\newtheorem{theorem}{Theorem}[section] 
\newtheorem{claim}[theorem]{Claim}
\newtheorem{lemma}[theorem]{Lemma} 
 
\newtheorem{conclusion}[theorem]{Conclusion}
 
\newtheorem{corollary}[theorem]{Corollary} 

\theoremstyle{definition}
\newtheorem{definition}[theorem]{Definition}
\newtheorem{example}[theorem]{Example}

\newtheorem{answer}[theorem]{Answer}

\theoremstyle{remark}
\newtheorem{remark}[theorem]{Remark}
\newtheorem{question}[theorem]{Question}
\newtheorem{notation}[theorem]{Notation}

\newcommand{\Th}{{\rm Th}}
\newcommand{\tp}{{\rm tp}}

\newcommand{\dpfor}{{\rm dpfor}}

\newcommand{\Av}{{\rm Av}}

\newcommand{\wilog}{{\rm without loss of generality}}
\newcommand{\Wilog}{{\rm Without loss of generality}}

\newcommand{\then}{{\underline{then}}}

\newcommand{\Then}{{\underline{Then}}}

\newcommand{\mn}{{\medskip\noindent}}
\newcommand{\sn}{{\smallskip\noindent}}

\newcommand{\bbR}{{\mathbb R}}

\newcommand{\cC}{{\mathscr C}}

\newcommand{\cF}{{\mathscr F}}

\newcommand{\cK}{{\mathscr K}}
\newcommand{\bbL}{{\mathbb L}}
\newcommand{\bbN}{{\mathbb N}}

\newcommand{\bbP}{{\mathbb P}}

\newcommand{\bbQ}{{\mathbb Q}}
\newcommand{\bbZ}{{\mathbb Z}}

\newcommand{\cS}{{\mathscr S}}
\newcommand{\cT}{{\mathscr T}}

\newcommand{\cf}{{\rm cf}}

\newcount\skewfactor
\def\mathunderaccent#1#2 {\let\theaccent#1\skewfactor#2
\mathpalette\putaccentunder}
\def\putaccentunder#1#2{\oalign{$#1#2$\crcr\hidewidth
\vbox to.2ex{\hbox{$#1\skew\skewfactor\theaccent{}$}\vss}\hidewidth}}

\begin{document}

\title {Perpendicular Indiscernible Sequences in Real Closed Fields }
\author {Eyal Firstenberg \and Saharon Shelah}
\address{Einstein Institute of Mathematics\\
Edmond J. Safra Campus, Givat Ram\\
The Hebrew University of Jerusalem\\
Jerusalem, 91904, Israel\\
 and \\
 Department of Mathematics\\
 Hill Center - Busch Campus \\ 
 Rutgers, The State University of New Jersey \\
 110 Frelinghuysen Road \\
 Piscataway, NJ 08854-8019 USA}
\email{shelah@math.huji.ac.il}
\urladdr{http://shelah.logic.at}
\thanks{This is essentially the MSc. thesis of the first author under
  the guidance of the second author.  Partially supported by Israel
Science Foundation. Publ.E50 in the second author's website.  We thank
Alice Leonhardt for the beautiful retyping of the work.}

\date{March 30, 2012}

\maketitle
\numberwithin{equation}{section}
\setcounter{section}{-1}
\newpage

\centerline {Annotated Contents}
\bigskip

\noindent
\S0 \quad Introduction, pg.
\bigskip

\noindent
\S1 \quad Preliminaries, pg.

\S(1.1) \quad A word about notation

\S(1.2) \quad Dependent Theories

\S(1.3) \quad Real Closed Fields
\bigskip

\noindent
\S2 \quad Perpendicularity in Dependent Theories, pg.
\bigskip

\noindent
\S3 \quad Cuts in Real Closed Fields, pg.
\bigskip

\noindent
\S4 \quad Perpendicularity in Real Closed Fields, pg.
\bigskip

\noindent
\S5 \quad Strong Perpendicularity in Real Closed Fields, pg.
\newpage

\section {Introduction}

In this paper we investigate the behaviour of concepts from dependent
theories when applied to real closed fields.  Our main focus will be
in the concept of perpendicular indiscernible sequences, a concept
first introduced in \cite[\S4]{Sh:715}.

A (first-order) theory is dependent when it's monster model doesn't
contain a sequence of finite sequences $\{\bar a_i\}_{i \in \omega}$
and a formula $\varphi(x,y)$ with $\ell(y) = \ell(\bar a_i)$ for $i < \omega$
 such that for every finite set $S
\subseteq \omega$ and a ``truth requirement function" $\eta:S
\rightarrow \{T,F\}$ there exists some $x_\eta$ with $\models
\bigwedge\limits_{i \in S} \varphi[x_\eta,\bar a_i]^{\eta(i)}$.  This
concept was introduced by Shelah, see \cite{Sh:715}.  This concept is
related to stability of models, and one can show that every stable
theory is also dependent.

Perpendicularity in dependent models (models of a dependent theory) is
a binary relation between infinite indiscernible sequences.  Two such
sequences $\bar{\bold a}^1 = \langle a^1_t:t \in I^1\rangle,\bar{\bold
  a}^2 = \langle a^2_t:t \in I^2\rangle$ are said to be perpendicular
if for every formula $\varphi(x,y)$ with $\ell(x) =
\ell(\alpha^1_t),\ell(y) = \ell(a^2_t)$ there exists a truth value $\bold
t$ such that for every large enough $t \in I^1$ for every large enough
$s \in I^2$ we have $\models \varphi[a_t,a_s]^{\bold t}$ and for every
large enough $s \in I^2$ for every large enough $t \in I^1$ we have
$\varphi[a_t,a_s]^{\bold t}$, see Definition \ref{15}.  This concept is
suggested in \cite[4]{Sh:715} as a substitute for orthogonal sequences
in stable theories.

A real closed field is an ordered field which exhibits the
Intermediate Value Theorem for polynomials, i.e. for every polynomial
$p(x)$ and elements $a,b$ such that $p(a) > 0,p(b) < 0$ there exists
some $c$ in the interval $(a,b)$ such that $p(c)=0$.  Tarski proved
that this theory has quantifier elimination, and so it is easily
concluded that the theory is dependent.  See ``Preliminaries".

In this paper we investigate perpendicularity of indiscernible
sequences in real closed fields.  In \S2 we properly define
perpendicularity and supply equivalent definitions.  We also
characterize perpendicular sequences based on results from
\cite{Sh:715}.  In \S3 we will review the subject of Dedekind
cuts in real closed Dedkind fields, emphasizing the notion of dependent cuts
from the model theoretic perspective.  In \S4 we connect the concepts of
perpendicularity of sequences and dependency of cuts in real closed
fields, and we show that under certain condition cuts are dependent
iff the sequences inducing them are not dependent (Claim \ref{46} and
Theorem \ref{42}).  In \S5 we define strong perpendicularity between
indiscernible sequences in dependent models, and we prove that in real
closed fields no two indiscernible sequences are strongly
perpendicular (Theorem \ref{63}).

We presume the reader is well familiar with model theory and abstract
model theory and fairly familiar with real closed fields.  A good
introduction to model theory can be found in \cite{Ho97}.  A good
introduction to real closed fields can be found in \cite{MMP96}.  We
also recommend the reader to familiarize himself or herself 
with the definitions and results which appeared in \cite{Sh:715} 
and \cite{Sh:783}, although we
will state and prove every definition or claim we use from there.
\newpage

\section {Preliminaries}

In this section we introduce the basic concepts on which the rest of the
work depends and several basic known results.  The section is divided into
two parts.  In the first part we define the concept of ``dependent
theories" and dependent formulas.  In the second part we cite Tarski's
theorem of quantifier elimination in the theory of real closed field,
and several immediate conclusions which will be helpful in the next
items.  The reader is invited to skip this section and use it as a
reference.
\bigskip

\subsection {A Word About Notation} \
\bigskip

We shall work mainly with the theory of real closed fields.  However,
when the theory under discussion is known we will denote by $\cC$ the
``monster model".  As usual, $\varphi^0$ means $\neg \varphi$ and
$\varphi^1$ mens $\varphi$.
\bigskip

\subsection {Dependent Theories}\
\bigskip

\begin{definition}  
\label{1}
Let $T$ be a complete theory.  $T$ is said to be independent if for
some formula $\varphi(\bar x,\bar y)$, possibly with parameters,
$\varphi(\bar x,\bar y)$ is independent, which means that for every $n
< \omega$:

\[
T \vdash \exists \bar y_0 \exists \bar y_1 \ldots \exists \bar
y_{n-1} \bigwedge\limits_{\eta \in \{0,1\}^n} \exists \bar x
(\bigwedge\limits_{k<n} \varphi(\bar x,\bar y_k)^{\eta(k)}).
\]

\mn
$T$ is said to be dependent if it's not independent.  A model is
called dependent if its theory is.
\end{definition}

\begin{example}  
\label{2}
1) Let $\langle V,E\rangle$ be a $*$-{\bf random graph}.  Then
   $\Th(\langle V,E\rangle)$ is independent.

\noindent
2) The theory of dense linear order is dependent.
\end{example}

\begin{proof}  
1) A $*$-random graph $\langle V,E\rangle$ is an infinite graph
$|V| \ge \aleph_0$ where every finite graph can be embedded into it.
   Here $\varphi(x,y) = xEy$ is independent.

\noindent
2) The theory of dense linear order can be interpreted in the theory
   of real closed fields.  Now use the next Lemma and Claim \ref{7}.
\end{proof}

\begin{lemma}  
\label{3}
Let $M$ be some dependent model and let $N$ be a model interpreted in
$M$.  Then $N$ is dependent.
\end{lemma}

\begin{proof}  
Denote by $\rho:N \rightarrow M^k,\rho:\tau(N) \rightarrow \bbL(M)$
the interpretation function.  So $\rho$ induces a function
$\rho':\bbL(N) \rightarrow \bbL(M)$ (and we use the notation
$\varphi_\rho = \rho'(\varphi))$ such that $N \models \psi[\bar a]
\Leftrightarrow M \models \psi_\rho[\rho(\bar a)]$.  Now assume by
contradiction that the claim is {\it false}.  So for some formula
$\varphi(\bar x,\bar y) \in \bbL(N)$ we have that $\rho$ is
independent.  We claim that $\varphi_\rho$ is independent in $M$.  Let
$n < \omega$.  By assumption there exists $\bar y_0,\dotsc,\bar
y_{n-1} \in {}^{\ell g(\bar y)}N$ such that

\[
N \models \bigwedge\limits_{\eta \in \{0,1\}^n} \exists \bar
x(\bigwedge\limits_{k<n} \varphi(\bar x,\bar y_k)^{\eta(k)}).
\]

\mn
So we have

\[
M \models \bigwedge\limits_{\eta \in \{0,1\}^n} \exists \bar x'
(\bigwedge\limits_{k<n} \varphi_\rho(\bar x',\rho(\bar y_k))^{\eta(k)}).
\]

\mn
Where $\ell g(\bar x') \cdot k$.  This shows that $\varphi_\rho$ is
independent in $M$.  Contradiction.
\end{proof}

\begin{definition}  
\label{4}
Let $I$ be some linearly ordered set and let $\bar{\bold a} = \langle
\bar a_t:t \in I\rangle$ be some indiscernible sequence with $\ell
g(\bar a_t) = m$ for every $t \in I$.  A formula $\varphi(\bar x,\bar
y)$ possibly with parameters, $\ell g(\bar y) = m$ is said to be
dependent relative to $\bar{\bold a}$, if for every $\bar b \in
{}^{\ell g(\bar x)} \cC$ the set $\{t \in I :\models 
\varphi[\bar b,\bar a_t]\}$ is a finite union of convex subsets of $I$.  We
define $\dpfor(\bar{\bold a}) \subseteq \bbL(T)$ 
to be the set of such formulas.
\end{definition}

\begin{claim}  
\label{5}
Let $T$ be a dependent theory, and $\bar{\bold a} = \langle \bar a_t:t
\in I\rangle$ some indiscernible sequence.  Then $\dpfor(\bar{\bold
  a}) = \{\varphi(\bar x,\bar y) \subseteq \bbL(T):\ell g(\bar y) =
\ell g(\bar a_t)\}$.  Furthermore, for every $\varphi(\bar x,\bar y)$
there exists $k_\varphi \in \omega$ such that for every indiscernible
sequence $\bar{\bold a} = \langle \bar a_t:t \in I\rangle$ and every
$\bar b \in {}^{\ell g(\bar x)}\cC$ we have that $\{t \in I :\models
\varphi[\bar b,\bar a_t]\}$ is a union of no more than $k_\varphi$
convex subsets of $I$.
\end{claim}

\begin{proof}  
$T$ is dependent and so by definition there exists $n<\omega$ such
  that:

\[
T \models \neg \exists \bar y_0 \exists \bar y_1 \ldots \exists \bar
y_{n-1} \bigwedge\limits_{\eta \in \{0,1\}^n} \exists \bar
x(\bigwedge\limits_{k<n} \varphi(\bar x,\bar y_k)^{\eta(k)}).
\]

\mn
Now assume toward contradiction $\bar{\bold a} = \langle \bar a_t:t
\in I\rangle$ is an indiscernible sequence and $\bar b \in {}^{\ell
  g(\bar x)}\cC$ such that $\langle \varphi{\bar b,\bar a_t}:t \in
I\rangle$ change signs $2n-1$ times.  \Wilog \, there exist $t_0 <_i
t_1 <_I \ldots <_I t_{2n-1}$ such that $\models \varphi[\bar b,\bar
  a_{t_k}]$ if $k$ odd.  Now according to the inset equation above 
there exists some $\eta
\in \{0,1\}^n$ such that $\models \forall \bar x(\neg
\bigwedge\limits_{k<n} \varphi(\bar x,\bar a_{t_k})^{\eta(k)})$, using
indiscernibility $\models \forall \bar x(\neg 
\bigwedge\limits_{k<n} \varphi(\bar x,
\bar a_{t_{2k+\eta(k)}})^{\eta(k)})$, in contradiction to 
$\models \bigwedge\limits_{k<n} \varphi[\bar b,
\bar a_{t_{2k +\eta(k)}}]^{\eta(k)}$, so $k_\varphi = 2n-1$ satisfies
the requirements.
\end{proof}
\bigskip

\subsection {Real Closed Field} \
\bigskip

We assume some background in the definitions of real closed fields.
Throughout this paper, $\cF$ will denote a model of a real closed
field.

We shall use the next theorem by Tarski quite often.  We bring it here
without proof.

\begin{theorem}  
\label{6}
(Quantifier elimination for the theory of real closed fields)

The theory of real closed fields has quantifier elimination in the
language $\langle +,\cdot,0,1,< \rangle$.
\end{theorem}

\begin{proof}  
See \cite[Ch.1]{MMP96}.
\end{proof}

\begin{claim}  
\label{7}
The theory of $\langle \bbR,+,\cdot,0,1,<\rangle$, i.e. the theory of
real closed fields is dependent.
\end{claim}

\begin{proof}
We work in $\langle \bbR,+,\cdot,0,1,<\rangle$.  By the previous
theorem, we reduce to the following case.

Assume for a contradiction that for some polynomial $p$ in 
$s+m$ variables we have
$p(\bar x,\bar y,\bar c) >0$ is independent with $\ell g(\bar x) =
s,\ell g(\bar y)=m$.  Rewrite $p(\bar x,\bar y) =
\sum\limits^{k-1}_{i=0} m_i(\bar x,\bar y,\bar c)$ where the $m_i$'s
are monomials.  By definition of independent formula there exist $\bar
b_0,\dotsc,\bar b_k$ such that for any function $\eta$ with domain
$\{0,\dotsc,k\}$ and range $\{0,1\}$,

\[
\models \exists \bar x(\bigwedge\limits_{i \le k} p(\bar x,\bar
b_i,\bar c) \cdot (-1)^{\eta(i)} > 0).
\]

\mn
Denote by $d^j_i$ the coefficient of $\bar x$ in the monomial
$m_i(\bar x,\bar b_j,\bar c)$ and $\bar d^j =
(d^j_0,d^j_1,\dotsc,d^j_{k-1})$.  Now by counting dimensions we know
that $\{\bar d^j:j \in \{0,\dotsc,k\}\}$ is linearly dependent over $\bbR$,
namely $\sum\limits^k_{j=0} a_j \bar d^j=0$, with $(a_0,\dotsc,a_k)
\ne \bar 0$.  So we have $\sum\limits^{k}_{j=0} a_j \cdot p(\bar
x,\bar b_j,\bar c) = 0$.  Choose $\eta$ as above so that
$\eta(j) = 0$ iff $a_j > 0$.  Now we assumed that for each such $\eta$
there is an $\bar x_*$ such that

\begin{equation*}
\begin{array}{clcr}
\bigwedge\limits_{j \le k} &p(\bar x_*,\bar b_j,\bar c) \cdot (-1)^{\eta(j)} >
  0 \text { thus} \\
& \bigwedge\limits_{j \le k} p(\bar x_*,\bar b_j,\bar c) \cdot a_j \ge 0
  \text{ with some polynomial strictly greater than } 0 \\
& \sum\limits_{j \le k} a_j \cdot p(\bar x_*,\bar b_j,\bar c) > 0.
\end{array}
\end{equation*}

\mn
This contradicts the definition of $(a_0,\dotsc,a_*)$ so we finish.
\end{proof}

\noindent
We shall use the following corollaries of quantifier elimination later on:
\begin{corollary}
\label{8}
Let $\cF$ be some real closed field and $\phi(x,\bar a)$ a formula
with parameters from $\cF$.  Then the set $\{x \in \cF :\models
\varphi[x,\bar a]\}$ is a finite union of intervals of $\cF$.
Furthermore, the number of intervals is bounded uniformly
independent of $\bar a$.  
\end{corollary}

\begin{corollary}
\label{9}
Let $f(x):\cF \rightarrow \cF$ be some function definable over $\cF$.
Then $f$ is piecewise monotonic, with each piece either constant or
strictly monotonic.  In other words, there exist $a_0 < \ldots <
a_{n+1}$ with $a_0 = -\infty,a_{n+1} = \infty$ such that $f$ is
constant or strictly monotonic on $(a_i,a_{i+1})$ for $i \in \{0,\dotsc,n\}$. 
\end{corollary}

\begin{corollary}
\label{10}
Let $\cF$ be some real closed field and $a,b \in \cF$ realize the same
type over some $A \in \cF$.  Then for all $d \in (a,b)_{\cF}$ we have
that $d$ realizes the same type over $A$ as $a,b$.
\end{corollary}

\begin{proof}
Assume otherwise and let $\varphi(x)$ witness: $\models (\neg
\varphi[d]) \wedge \varphi[a]$.  By the previous corollaries $\varphi$
divides the field into finitely many intervals, so define $\psi_n(x) =
``\varphi(x)$ and $x$ is in the $n$-th interval realizing
$\varphi(x)$".  It is easy to see that the $\psi_n$'s use only
parameters from $A$ so for some $n$ we have that $\models \psi_n[a]$.
So by assumption $\models \psi_n[b]$ this contradicts the existence of
$d \in (a,b)_{\cF}$.
\end{proof}
\newpage

\section {Perpendicularity in Dependent Theories} 

In this item we define and explore the notion of perpendicular
indiscernible sequence in dependent theories.  This notion is somewhat
parallel to the notion of orthogonal sequences in stable theories (see
\cite[Ch.V]{Sh:c}).  We also introduce a technique for constructing
indiscernible sequences based on ultra-filters or other indiscernible
sequences.  This technique is very important to the understanding of
the rest of the work.

\relax From here on we will denote by $T$ a dependent theory.
\bigskip

\begin{definition}
\label{11}
1) Let $\bar{\bold a} = \langle \bar a_t:t \in I\rangle$ be an
\underline{endless} indiscernible sequence (i.e. $I$ has no last
element).  For every $A \subset \cC$ and $\Delta
   \subseteq \bbL(T)$ we define the $\Delta$-average of $\bar{\bold
   a}$ over $A$ or $\Av_\Delta(A,\bar{\bold a})$ to be the type:

\begin{equation*}
\begin{array}{clcr}
\Av_\Delta(A,\bar{\bold a}) = \{\varphi(\bar x,\bar b)^{\bold
  t}:&\varphi(\bar x,\bar y) \in \Delta,\bar b \in {}^{\ell g(\bar
  y)}A,\ell g(\bar x) = \ell g(\bar a_t), \\
  &\text{ for every large enough } t \in I \text{ we have } \models
  \varphi[\bar a_t,\bar b]^{\bold t}\}.
\end{array}
\end{equation*}

\mn
2) Let $D$ be some ultrafilter on some ${}^m C \subset \cC$.  For every
$A \subset \cC$ we define the $\Delta$-average of $D$ over $A$ or
$\Av_\Delta(A,D)$ to be the type:

\begin{equation*}
\begin{array}{clcr}
\Av_\Delta(A,D) = \{\varphi(\bar x,\bar b)^{\bold t}:&\varphi(\bar
 x,\bar y) \in \Delta,\bar b \in {}^{\ell g(\bar y)}A,
\ell g(\bar x) = m, \\
  &\{\bar c \in {}^m \cC:\models \varphi[\bar c,\bar b]^{\bold t}\}
 \in D.\}
\end{array}
\end{equation*}
 
\mn
Omitting the $\Delta$ means for $\Delta = \bbL(T)$.

In Clause 2 of the above definition we think of ${}^m \cC$ as a ``test
set" and we say $\varphi(\bar x,\bar b) \in \Av_\Delta(A,D)$ if the
``majority" of ${}^m C$, according to $D$, exhibit $\models
\varphi[\bar x,\bar b]$.  Since $D$ is closed under finite
disjunctions, $\Av_\Delta(A,D)$ is finitely realized in ${}^m C$.
\end{definition}

\begin{remark}
Note that in Definition \ref{11} when writing $\Av_\Delta(A,\bar{\bold
  a})$ the parameter set $A$ comes before the object $\bar{\bold a}$
  in question, in keeping with usage in published articles of the
  second author.
\end{remark}

\begin{conclusion}
\label{12}
If $T$ is dependent and $I$ is endless, then $\Av(A,\bar{\bold a})$ is
a complete type over $A$.
\end{conclusion}

\begin{proof}
Follows from Claim \ref{5}.
\end{proof}

\begin{definition}
\label{13}
Let $T$ be dependent, $\bar{\bold a}$ a $\Delta$-indiscernible
sequence in $\cC$ and $A \subset \cC$ with $\bar{\bold a} \subseteq
A$.  We say that $\bar{\bold b} = \langle \bar b_t:t \in I\rangle$ is
$\Delta$-based on $\bar{\bold a}$ over $A$ if for every $t \in I$ we
have that $\bar b_t$ realizes $\Av_\Delta(A \cup \{\bar b_s:s <_I
t\},\bar{\bold a})$.  Omitting the $\Delta$ means for $\Delta = \bbL(T)$.
\end{definition}

\begin{example}
Let $T$ be the theory of dense linear order with a model $\bbQ$ and
let $\bar{\bold a} = \langle 10 - \frac 1n:n \in \bbN\rangle$.  We
will find a sequence based on $\bar{\bold a}$ over $\bbQ$ in some real
closed field extending $\bbQ$.  To that end, we will need to find an
element which realize $\{0 < x < p:0 < p \in \bbQ\}$, denoted
$\delta$.  Now the reader can check that the sequence $\langle 10 - n
\delta:n \in \bbN\rangle $ is as needed.
\end{example}

\begin{claim}
\label{14}
1) If $\bar{\bold b}^1,\bar{\bold b}^2$ are $\Delta$-based on
$\bar{\bold a} = \langle \bar a_s:s \in J\rangle$ over some $A$ with
the same order type $I$, then $\bar{\bold b}^1,\bar{\bold b}^2$ have
 the same $\Delta$-type over $A$.

\noindent
2) If $\bar{\bold b}$ is $\Delta$-based on $\bar{\bold a}$ over some
   $A$ then $\bar{\bold b}$ is $\Delta$-indiscernible over $A$.
\end{claim}

\begin{remark}
\label{r1}
In the Definition \ref{13} of $\Delta$-based we require $\bar{\bold a}
\subseteq A$.
\end{remark}

\begin{proof}
1) Assume toward contradiction that there exists $\varphi \in
   \Delta,t_1 <_I \ldots <_I t_{k_\varphi},\bar d_\varphi \subseteq A$
   with $\models \varphi[\bar b^1_{t_1},\dotsc,\bar
   b^1_{t_{k_\varphi}},\bar d_\varphi]$ but $\models \neg \varphi[\bar
   b^2_{t_1},\dotsc,\bar b^2_{t_{k_\varphi}},\bar d_\varphi]$.  Take
   such $\varphi$ with minimal $k_\varphi$.  By minimality of
   $k_\varphi$ we have that $\bar b^1_{t_1},\dotsc,\bar
   b^1_{t_{k-1}},\bar b^2_{t_1} \ldots \bar b^2_{t_{k-1}}$ realize the
   same $\Delta$-type over $A$.  By definition of average we know that
   for every large enough $s \in J$ we have that
$\models \varphi[\bar b^1_{t_1},\dotsc,\bar b^1_{t_{k-1}},\bar
   a_s,\bar d_\varphi]$, hence also for every large enough $s \in J$
   we have $\models \varphi[\bar b^2_{t_1},\dotsc,\bar b^2_{t_{k-1}},
\bar a_s,\bar d_\varphi]$.  Now by the definition of average we get
$\models \varphi[\bar b^2_{t_1},\dotsc,\bar b^2_{t_{k-1}},
\bar b^2_{t_k},\bar d_\varphi]$, contradicting the assumptions.

\noindent
2) Follows from Clause 1 noticing that every finite subsequence of
   $\bar{\bold b}$ is actually a sequence based on $\bar{\bold a}$
   over $A$.
\end{proof}

\begin{definition}
\label{15}
1) We say that two sequences $\bar{\bold b}^1,\bar{\bold b}^2$ are
   $\Delta$-mutually indiscernible over $A$ if for $\ell=1,2$ we have
   that $\bar{\bold b}^\ell$ is $\Delta$-indiscernible over
   $\bar{\bold b}^{3-\ell} \cup A$.

\noindent
2) Omitting $A$ means for $A = \phi$.  Omitting $\Delta$ means $\Delta
   = \bbL(T)$.

\noindent
3) We say that two endless $\Delta$-indiscernible sequences
   $\bar{\bold a}^1,\bar{\bold a}^2$ are $\Delta$-perpendicular iff
   for every $A \subset \cC$, for some $\langle \bar b^\ell_n:\ell \in
   \{1,2\},n < \omega\rangle$ where $\bar b^\ell_n$ realizes
   $\Av_\Delta(A \cup \bar{\bold a}^1 \cup \bar{\bold a}^2 \cup \{\bar
   b^k_m:k \in \{1,2\},m < n$ or $(m = n \wedge k < \ell)\},
\bar{\bold a}^\ell)$ for $\ell \in \{1,2\},n < \omega$ we have that $\langle
   \bar b^1_n:n < \omega\rangle,\langle \bar b^2_n:n < \omega\rangle$
   are $\Delta$-mutually indiscernible over $A \cup \bar{\bold a}^1
   \cup \bar{\bold a}^2$.  Omitting the $\Delta$ means for $\Delta = \bbL(T)$.
\end{definition}

\begin{example}
\label{16}
In the theory of dense linear order, two strictly increasing
indiscernible sequences of elements are perpendicular iff they induce
different cuts (see below).
\end{example}

\begin{claim}
\label{17}
In Definition 15, Clause 3 replacing ``for some $\langle \bar
b^\ell_n:\ell \in \{1,2\},n < \omega\rangle$" with ``for every
$\langle \bar b^\ell_n:\ell \in \{1,2\},n < \omega\rangle$" is equivalent.
\end{claim}

\begin{proof}
We will deal with the non-trivial direction only.  It is enough to
prove that for any two pairs of sequences $\langle \bar
b^{\ell,i}_n:\ell \in \{1,2\},n < \omega\rangle$, constructed as in
Definition \ref{15}(3), i.e. with $\bar b^{\ell,i}_n$ realizing
$\Av_\Delta(A \cup \bar{\bold a}^1 \cup \bar{\bold a}^2 \cup \{\bar
b^{k,i}_m:k \in \{1,2\},m < n$ or $(m = n \wedge k <
\ell)\},\bar{\bold a}^\ell)$, we have that $\langle \bar b^{\ell,i}_n:
\ell \in \{1,2\},n < \omega\rangle$ realize the same type over $A \cup
\bar{\bold a}^1 \cup \bar{\bold a}^2$ for $i \in \{1,2\}$.  Define a
new order $\prec$ on these sequences: $\bar b^{\ell,i}_n \prec \bar
b^{k,i}_m$ iff $n < m$ or $(m = n \wedge \ell < k)$.  From here the
proof is similar to the proof of Claim \ref{14}.
\end{proof}

\begin{claim}
\label{18}
Being perpendicular and being $\Delta$-perpendicular are both
symmetric notions.
\end{claim}

\begin{proof}
Assume $\bar{\bold a}^1,\bar{\bold a}^2$ are $\Delta$-perpendicular
and we will prove that $\bar{\bold a}^2,\bar{\bold a}^1$ are
$\Delta$-perpendicular.  So let $\langle \bar b^\ell_n:\ell \in
\{1,2\},n < \omega\rangle$ be such that $\bar b^\ell_n$ realizes

\[
\Av_\Delta(A \cup \bar{\bold a}^1 \cup \bar{\bold a}^2 \cup \{b^k_m:k
\in \{1,2\},m < n \text{ or } (m=n \wedge k < \ell)\},\bar{\bold
  a}^\ell)
\]

\mn
for $\ell \in \{1,2\},n < \omega$.  By Claim 17 $\langle b^\ell_n:n <
\omega\rangle,\langle \bar b^2_n:n < \omega\rangle$ are
$\Delta$-mutually indiscernible.  Define $\bar c^1_n = \bar b^1_{n+1}$
for $n < \omega$ and $\bar c^2_n = \bar b^2_n$ for $n < \omega$.  Now
easily $\langle \bar c^2_n:n < \omega\rangle,\langle \bar c^1_n:n <
\omega\rangle$ are witnesses that $\bar{\bold a}^2,\bar{\bold a}^1$
are $\Delta$-perpendicular.
\end{proof}

\begin{claim}
\label{19}
Assume $\bar{\bold a}^1 = \langle \bar a^1_t:t \in
I^1\rangle,\bar{\bold a}^2 = \langle \bar a^2_s:s \in I^2\rangle$ are
$\Delta$-indiscernible and $J^\ell \subseteq I^\ell$ is unbounded in
$I^\ell$ for $\ell \in \{1,2\}$.  So $\bar{\bold c}^1 = \langle \bar
a^1_t:t \in J^1\rangle,\bar{\bold c}^2 = \langle \bar a^2_s:s \in
J^2\rangle$ are $\Delta$-perpendicular iff $\bar{\bold a}^1,\bar{\bold
  a}^2$ are $\Delta$-perpendicular.
\end{claim}

\begin{proof}
This follows from the following fact: if $\langle \bar a_t:t \in
I\rangle$ is $\Delta$-indiscernible and $J \subseteq I$ is unbounded,
then for every $A \subset \cC$ we have $\Av_\Delta(A,\langle \bar
a_t:t \in I\rangle) = \Av_\Delta(A,\langle \bar a_t:t \in J\rangle)"$.
\end{proof}

\begin{claim}
\label{20}
Let $\langle a^1_t:t \in I^1\rangle,\langle \bar a^2_t:t \in
I^2\rangle$ be two indiscernible sequences.  Then for every formula
$\varphi(\bar x,\bar y)$ possibly with parameters there exists a truth
value $\bold t$ such that for every large enough $t \in I^1$ for every
large enough $s \in I^2$ we have $\models \varphi[\bar a^1_t,\bar
  a^2_s]^{\bold t}$.
\end{claim}

\begin{proof}
Assume the contrary.  Then by Claim \ref{5} for any $n \in \omega$ there
exists an increasing sequence $t_0 <_{I^1} \ldots <_{I^1} t_{n-1}$
such that for any $j =0,\dotsc,n-1$ for every large enough $s \in I^2$
we have that $\models \varphi[\bar a^1_{t_j},\bar a^2_s]$ iff $j$
odd.  $n$ is finite and therefore we can find $s \in I$ large enough
for every $j<n$.  Hence for every $n \in \omega$ there exists some $s
\in I^2$ with $\{t \in I^1: \models \varphi[\bar a^1_t,\bar a^2_s]\}$
contains no less than $\frac n2$ convex subsets of $I^1$,
contradicting the Claim \ref{5}.
\end{proof}

\begin{remark}
While Claim \ref{20} is symmetric, note that in Claim \ref{21}(4) we
choose $\bold t$ in advance which must cohere for both directions.
\end{remark}

\begin{claim}
\label{21}
Let $\bar{\bold a}^1 = \langle \bar a^1_t:t \in I^1\rangle,\bar{\bold
  a}^2 = \langle \bar a^2_t:t \in I^2\rangle$ be two indiscernible
  sequences.  \Then \, the following are equivalent:
\mn
\begin{enumerate}
\item  $\bar{\bold a}^1,\bar{\bold a}^2$ are perpendicular
\sn
\item  for every $\langle \bar b^\ell_n:\ell=1,2;n \in \omega\rangle$
  such that each $\bar b^\ell_n$ realizes
\newline
$\Av(\bar{\bold a}^1 \cup \bar{\bold a}^2 \cup\{\bar b^k_m:k \in
\{1,2\},m<n$ or $(m=n \wedge k<\ell)\},\bar{\bold a}^\ell)$
\newline
we have that $\langle \bar b^1_n:n \in \omega\rangle,\langle \bar
b^2_n:n \in \omega\rangle$ are mutually indiscernible
\sn
\item  for every $A \subset \cC$ if $\bar b^1$ realizes
  $\Av(\bar{\bold a}^1,\bar{\bold a}^1 \cup \bar{\bold a}^2 \cup A)$
  and $\bar b^2$ realizes $\Av(\bar{\bold a}^2,
\bar{\bold a}^1 \cup \bar{\bold a}^2 \cup A 
\cup \bar b^1)$ then $\bar b^1$ also realizes
$\Av(\bar{\bold a}^1,\bar{\bold a}^1 \cup \bar{\bold a}^2 \cup A \cup b^2)$
\sn
\item for every formula $\varphi(\bar x,\bar y)$ possibly with
  parameters there exists a truth value $\bold t$ such that for every
  large enough $t \in I^1$ for every large enough $s \in I^2$ we have
  $\models \varphi[\bar a^1_t,\bar a^2_s]^{\bold t}$ and for every
  large enough $s \in I^2$ for every large enough $t \in I^1$ we have
  $\models \varphi[\bar a^1_t,\bar a^2_s]^{\bold t}$
\sn
\item  for some $(|\bar{\bold a}^1| + |\bar{\bold a}^2|)^+$-saturated
  model $\cF$ containing $\bar{\bold a}^1,\bar{\bold a}^2$ for every
formula $\varphi(\bar x,\bar y)$ possibly with parameters from
  $\cF$ there exists a truth value $\bold t$ such that for every large
  enough $t \in I^1$ for every large enough $s \in I^2$ we have
  $\models \varphi[\bar a^1_t,\bar a^2_s]^{\bold t}$ and for every
  large enough $s \in I^2$ for every large enough $t \in I^1$ we have
  $\models \varphi[\bar a^1_t,\bar a^2_s]^{\bold t}$.
\end{enumerate}
\end{claim}

\begin{proof}
1) We will show $(1) \rightarrow (2) \rightarrow (3) \rightarrow (4)
\rightarrow (5),(5) \rightarrow (4) \rightarrow (3) \rightarrow (1)$,
which suffices.

\noindent
2) Follows from (1) by choosing $A = \phi$ in Definition \ref{15}(3) and
Claim \ref{17}.

Now assume (2) and by contradiction assume that (3) fails for some $A$.
Now let $\langle \bar b^\ell_n:\ell=1,2;n \in \omega\rangle$ be such
that $\bar b^\ell_n$ realizes $\Av(A \cup \bar{\bold a}^1 \cup
\bar{\bold a}^2 \cup \{\bar b^k_m:k \in \{1,2\},m<n$ or $(m=n \wedge k
< \ell)\}$ for $\ell = 1,2,n \in \omega$.  (3) fails, so there exists
some formula $\varphi(\bar x,\bar y)$ with parameters from $A \cup
\bar{\bold a}^1 \cup \bar{\bold a}^2$ such that $\models \varphi[\bar
  b^1_k,\bar b^2_\ell]$ iff $k > \ell$.  Now denote $\psi(\bar x,\bar
y,\bar c) = \varphi(\bar x,\bar y)$ such that $\psi(\bar x,\bar y,\bar
z)$ is parameter free.  We claim that $\psi(\bar x,\bar y,\bar z)$ is
independent.  So for every $n \in \omega$ and $\eta \in {}^n 2$ we can
choose $k_0 < \ldots < k_{n-1}$ and $\ell_0 < \ldots < \ell_{n-1}$
such that for every $j=0,\dotsc,n-1$ we have $k_j < \ell_j$ iff
$\eta(j)=0$.  Now $\models \bigwedge\limits_{j<n} \psi(\bar
b^1_{k_j},\bar b^2_{\ell_j},\bar c)^{\eta(j)}$ and by
mutual-indiscernibility $\models \exists \bar z \bigwedge\limits_{j<n}
\psi(\bar b^1_j,\bar b^2_j,\bar z)^{\eta(j)}$ so the formula
$\psi(\bar x,\bar y,\bar z)$ is independent, contradiction.

To prove (4) from (3) let $\varphi(\bar x,\bar y,\bar c)$ be some formula
with $\bar c \subseteq A$.  Now let $\bar b^1$ realize 
$\Av(\bar{\bold a}^1 \cup \bar{\bold a}^2 \cup A,\bar{\bold a}^1)$ 
and let $\bar b^2$ realize $\Av(\bar{\bold a}^1 
\cup \bar{\bold a}^2 \cup A \cup
\bar b^1,\bar{\bold a}^2)$.  By Claim \ref{20} there exist truth values
$\bold t^1,\bold t^2$ such that for every large enough $t \in I^1$ for
every large enough $s \in I^2$ we have $\models \varphi[\bar
  a^1_t,\bar a^2_s,\bar c]^{\bold t^1}$ and for every large enough $s
\in I^2$ for every large enough $t \in I^1$ we have $\models
\varphi[\bar a^1_t,\bar a^2_s,\bar c]^{\bold t^2}$.  We will prove
that $\bold t^1 = \bold t^2$.  Now for every large enough $s \in I^2$
we have that $\models \varphi[\bar b^1,\bar a^2_s,\bar c]^{\bold
  t^2}$.  By definition of average $\models \varphi[\bar b^1,\bar
  b^2,\bar c]^{\bold t^2}$.  Now by (3) we have that $\bar b^1$ realize
$\Av(\bar{\bold a}^1 \cup \bar{\bold a}^2 \cup A \cup \bar
b^2,\bar{\bold a}^1)$, hence for every large enough $t \in I^1$ we
have $\models \varphi[\bar a^1_t,\bar b^2,\bar c]^{\bold t^2}$.  So
for every large enough $t \in I^1$ for every large enough $s \in I^2$
we have $\models \varphi[\bar a^1_t,\bar a^2_s,\bar c]^{\bold t^2}$.
So $\bold t^1 = \bold t^2$ and we are done.  

\noindent
(3) Follows from (4) is symmetric.

\noindent
(5) Follows from (4) trivially.

To prove (4) from (5) assume toward contradiction that (4) fails.  So some
formula $\varphi(\bar x,\bar y,\bar d)$ witness the failure of (4).  Now
let $\bar e \subset \cF$ realize $\tp(\bar d,\bar{\bold a}^1 \cup
\bar{\bold a}^2)$.  So $\varphi(\bar x,\bar y,\bar e)$ witness the
failure of (5).  Contradiction and so (4) follows from (5).

To (1) from (3) it is enough to notice that by (3) $\langle \bar b^1_n:n \in
\omega\rangle$ constructed as in Definition \ref{15}(3) is actually a
sequence based on $\bar{\bold a}$ over $A \cup \langle \bar b^2_n:n
\in \omega\rangle$ and vice-versa.  Hence by Claim \ref{14} and Remark
\ref{r1} the two sequences are mutually indiscernible and we are done.
\end{proof}

\noindent
We now introduce a weaker version of perpendicularity to use in the
next section.
\begin{definition}
\label{22}
Two indiscernible sequences $\bar{\bold a}^1 = \langle \bar a^1_t:t
\in I^1\rangle,\bar{\bold a}^2 = \langle \bar a^2_s:s \in I^2\rangle$
will be called $(\Delta,A)$-perpendicular where $\Delta$ is some set
of formulas of the form $\varphi(\bar x,\bar y,\bar z)$ with $\ell
g(\bar x) = \ell g(\bar a^1_t),\ell g(\bar y) = \ell g(\bar a^2_s)$
and $\bar{\bold a}^1 \cup \bar{\bold a}^2 \subseteq A$ iff for every
$\varphi(\bar x,\bar y,\bar z) \in \Delta,\bar d \subset A$ with $\ell
g(\bar d) = \ell g(\bar z)$ for some truth value $\bold t$ we have
that for every large enough $g \in I^1$ for every large enough $s \in
I^2$ we have $\models \varphi[\bar a^1_t,\bar a^2_s,\bar d]^{\bold t}$
and for every large enough $s \in I^2$ for every large enough $t \in
I^1$ we have $\models \varphi[\bar a^1_t,\bar a^2_s,\bar d]^{\bold
  t}$.  Omitting the $\Delta$ means for $\Delta = \bbL(T)$.
\end{definition}

\noindent
As a corollary of \ref{21} we obtain:
\begin{corollary}
\label{23}
1) If $M$ is $(|\bar{\bold a}^1| + |\bar{\bold a}^2|)^+$-saturated
   then $\bar{\bold a}^1,\bar{\bold a}^2$ are
   $(\Delta,M)$-perpendicular iff they are $\Delta$-perpendicular.

\noindent
2) If $B \subseteq A,\Delta_B \subseteq \Delta_A$ and 
$\bar{\bold a}^1,\bar{\bold a}^2$ are $(\Delta_A,A)$-perpendicular
then they are $(\Delta_B,B)$-perpendicular.
\end{corollary}

\begin{claim}
\label{24}
Let $\bar{\bold a}^1 = \langle \bar a^1_t:t \in I^1\rangle,\bar{\bold
  a}^2 = \langle \bar a^2_t:t \in I^2\rangle$ be two $A$-perpendicular
  sequences.  Assume that for some $\bar{\bold b}^1 = \langle \bar
  b^1_n:n \in \omega \rangle,\bar{\bold b}^2 = \langle \bar b^2_n:n
  \in \omega\rangle$ we have that $\bar{\bold b}^1,\bar{\bold b}^2
  \subseteq A$ and for every $n \in \omega,\ell \in \{1,2\}$ we have
  that $b^\ell_n$ realizes

\[
\Av(\bar{\bold a}^1 \cup \bar{\bold a}^2 \cup \{\bar b^k_m:k \in
\{1,2\},m < n \text{ or } (m=n \wedge k < \ell)\},\bar{\bold a}^\ell).
\]

\mn
Then $\bar{\bold a}^1,\bar{\bold a}^2$ are perpendicular.
\end{claim}

\begin{proof}
By claim \ref{21}(2) it is enough to show that $\bar{\bold
  b}^1,\bar{\bold b}^2$ are mutually indiscernible.  So it is enough
to show that $\bar{\bold b}^1$ is based on $\bar{\bold a}^1$ over
  $\bar{\bold b}^2$ (the proof of the other direction is symmetric).
  This will follow if we show that for every $n \in \omega,m > n$ we
  have that $\bar b^1_n$ realize

\[
\Av(\bar{\bold a}^1 \cup \bar{\bold a}^2 \cup \{\bar b^1_k:k < n\}
\cup \{\bar b^2_k:k < m\},\bar{\bold a}^1).
\]

\mn
We will prove this by induction on $m \in \omega$.  So assume that
$\bar b^1_n$ realize

\[
\Av(\bar{\bold a}^1 \cup \bar{\bold a}^2 \cup \{\bar b^1_k:k < n\}
\cup \{\bar b^2_k:k < m\},\bar{\bold a}^1).
\]

\mn
Now we know that $\bar b^2_m$ realize

\[
\Av(\bar{\bold a}^1 \cup \bar{\bold a}^2 \cup \{\bar b^k_s:k \in
\{1,2\},s < m \text{ or } (s=m \wedge k=1)\},\bar{\bold a}^2).
\]

\mn
Let $\varphi(\bar x,\bar b^2_m,\bar c) \in \Av(\bar{\bold a}^1 \cup 
\bar{\bold a}^2 \cup \{\bar b^1_k:k < n\}$ or 
$\{\bar b^2_k:k < m+1\},\bar{\bold a}^1)$ with $\bar c \subseteq
\bar{\bold a}^1 \cup \bar{\bold a}^2 \cup \{\bar b^1_k:k < n\} \cup
\{\bar b^2_k:k < m\} \subseteq A$.  So for every large enough $t \in
I^1$ we have that $\models \varphi[\bar a^1_t,\bar b^2_m,\bar c]$.
Hence for every large enough $s \in I^2$ we have that $\models
\varphi[\bar a^1_t,\bar a^2_s,\bar c]$.  By $A$-perpendicularity we
have that for every large enough $s \in I^2$ for every large enough $t
\in I^1$ we have $\models\varphi[\bar a^1_t,\bar a^2_s,\bar c]$.  So
for every large enough $s \in I^2$ we have $\models \varphi[\bar
  b^1_n,\bar a^2_m,\bar c]$.  Hence $\models \varphi[\bar b^1_n,\bar
  b^2_m,\bar c]$.  So $\bar b^1_n$ realize $\Av(\bar{\bold a}^1 \cup
\bar{\bold a}^2 \cup \{\bar b^1_k:k < n\} \cup \{\bar b^2_k:k <
m+1\},\bar{\bold a}^1)$ and we are done.
\end{proof}

\begin{claim}
\label{25}
Let $\bar{\bold a}^1 = \langle \bar a^1_t:t \in I^1\rangle,\bar{\bold
  a}^2 = \langle \bar a^2_s:s \in I^2\rangle$ be two not
  $(\Delta,A)$-perpendicular sequences for some $\phi \ne A 
\subseteq \cC$. Then $\cf(\bar{\bold a}^1) = \cf(\bar{\bold a}^2)$.
\end{claim}

\begin{proof}
Denote $\lambda = \cf(\bar{\bold a}^1)$.  We will construct two
subsequence $\langle \bar a^1_{t_\alpha}:\alpha \in
\lambda\rangle,\langle \bar a^2_{s_\beta}:\beta \in \lambda\rangle$ of
$\bar{\bold a}^1,\bar{\bold a}^2$ respectively such that both are
unbounded.  This will prove that $\cf(\bar{\bold a}^1) =
\cf(\bar{\bold a}^2)$.  First take some unbounded subsequence $\langle
t'_\alpha:\alpha \in \lambda\rangle$ of $I^1$.  We construct the
sequences by induction on $\lambda$.  By definition of
perpendicularity we have that for some $\varphi(\bar x,\bar y,\bar c)
\in \Delta$;
\mn
\begin{enumerate}
\item   for every large enough $t \in I^1$ for every large enough $s
  \in I^2$ we have $\models \varphi[\bar a^1_t,\bar a^2_s,\bar c]$
\sn
\item  for every large enough $s \in I^2$ for every large enough $t
\in I^1$ we have $\models \varphi[\bar a^1_t,\bar a^2_s,\bar c]$.
\end{enumerate}
\mn
Now choose $t_0$ large enough as in (1) and larger than $t'_0$ and
choose $s_0$ large enough as in (1) for $t_0$ and also large enough as
in (2).  Now assume we constructed the sequences up to some $\alpha$.
Choose $t_{\alpha +1}$ large enough as in (2) for $s_\alpha$ and
larger than $t'_{\alpha +1}$ and choose $s_{\alpha +1}$ large enough
as in (1) for $t_{\alpha +1}$.  For limit ordinals $\alpha$ choose
$t_\alpha > \max_{i < \alpha}(t_i)$ and larger than $t'_\alpha$ (there
exists such $t_\alpha$ since $\cf(\bar{\bold a}^1) = \lambda >
\alpha$) and choose $s_\alpha$ large enough as in (1) for $t_\alpha$.
The construction is thus completed.

Now $\langle \bar a^1_{t_\alpha}:\alpha \in \lambda\rangle$ is
unbounded since $t_\alpha > t'_\alpha$ for every $\alpha \in \lambda$
and $\langle t'_\alpha:\alpha \in \lambda\rangle$ is unbounded.
Assume by contradiction that $\langle s_\alpha:\alpha \in
\lambda\rangle$ is bounded in $I^2$.  Then for some $s \in I^2$ we
have that $s$ is large enough as in (1) for every $t \in I^1$.  This
contradicts (2).  Then $\langle s_\alpha:\alpha \in \lambda\rangle$ is
unbounded and so $\cf(\bar{\bold a}^2) \le \cf(\bar{\bold a}^1)$ by
symmetry we are done.
\end{proof}
\newpage

\section {Cuts in Real Closed Fields}

In this section we explore cuts in real closed fields, more specifically
from the model theoretic point of view.  We define dependency of cuts and
review several equivalent definitions and results.  The last result in
this section (Claim \ref{39}) states that two cuts in a real closed field
$\cF$ are dependent if some polynomial $p(x,y) \ne 0$ with
coefficients from $\cF$ has a sequence of roots $\langle (a_t,b_t):t
\in I\rangle$ with $\langle a_t:t \in I\rangle,\langle b_t:t \in
I\rangle$ inducing the two cuts (for the definition of sequences
inducing cuts, see definition \ref{40}).  We will use this result in
subsequence sections.

\begin{definition}
\label{26}
1) Let $\cF$ be a real clsoed field.  A cut is a pair $C = (C^-,C^+)$
   such that $C^-,C^+ \subseteq \cF,\cC^-$ an initial segment of
   $\cF,C^+$ an end segment of $\cF$ and $C^- \cup C^+ = \cF,C^- \cap
   C^+ = \phi$.

\noindent
2) A cut $C = (C^+,C^-)$ will be called Dedekind if $C^-$ has no
   maximal element and $C^+$ has no minimal element.

\noindent
3) We define the cofinality of the cut $C$ to be the pair
   $(\cf(C^-),\cf(C^{+,*}))$ where $C^{+,*}$ is $C^+$ going
   backwards.  $\cf(C^-)$ will be called the left cofinality of $C$
   and right cofinality is defined in the same way.
\end{definition}

\begin{definition}
\label{27}
Let $\cF$ be a real closed field and let $C = (C^-,C^+)$ be a cut in
it.  Let $\cK$ be a real closed field extending $\cF$.  We say that $a
\in \cK$ realizes the cut $C$ if $c < a$ for every $c \in C^-$ and $a
< d$ for every $d \in C^+$. In this case we say that $\cK$ has
realization of $C$.
\end{definition}

\begin{definition}
\label{28}
Let $\cF$ be a real closed field and let $S$ be a family of cuts in
$\cF$.  We say that $S$ is dependent if for some real closed field
extending $\cF$ with realizations $\{a_s:s \in S\}$ of the cuts in $S$
we have that $\{a_s:s \in S\}$ is algebraically dependent over $\cF$,
i.e. there exists some $n < \omega$ and 
polynomial $p(\bar x) \ne 0$ in $n$ variables
with parameters from $\cF$ such that some sequence of length $n$ of
elements in $\{a_s:s \in S\}$ is a solution to $p(\bar x) = 0$.
Otherwise, the set is said to be independent.

The reader can easily check that this definition is sound.
\end{definition}

\begin{claim}
\label{29}
Let $\cF$ be a real closed field and $S$ a set of cuts in $\cF$.  
Then $S$ is independent iff whenever $D \subseteq S,\cK \ge \cF$ and
$\{a_d:d \in D\} \subseteq \cK$ realize the cuts in $D$ \then \, the
real closure of $\cF(\{a_d:d \in D\})$ does not realize any cut from
$S \backslash D$.
\end{claim}

\begin{proof}
First assume that $S$ is independent and let $D$  be a subset of $S$.
Take some $\cK$ a real closed field extending $\cF$ with realizations
$\{a_d:d \in D\}$ to the cuts in $D$.  We claim that the real closure
of $\cF(\{a_d:d \in D\})$ realizes no type in $S \backslash D$.  Assume
otherwise, then some cut $C$ in $S \backslash D$ is realized.  Hence
some polynomial $p(x)$ with parameters from $\cF(\{a_d:d \in D\})$
witness it.  Now $p(x)$ can be rewritten as $p(x,\bar d)$ with
parameters from $\cF$, where $\bar d$ is a sequence from $\{a_d:d \in
D\}$.  So $p(\bar x,\bar y)$ witness that $S$ is algebraically
dependent over $\cF$.  Contradiction.

For the second direction assume $S$ is dependent.  Then there exists
some polynomial $p(x) \ne 0$ with parameters from $\cF$ such that some
finite sequence from $\{a_s:s \in S\}$ solves $p(\bar x)=0$, assume
$p(\bar d) = 0,d = \langle d_1,\dotsc,d_n\rangle$ and $d_i$ realize
the cut $D_i$ in $S$ for $i=1,\dotsc,n$.  Without loss of generality
$p(x,d_2,\dots,d_n) \ne 0$ and so $D_1$ is realized in the real
closure of $\cF(d_2,\dotsc,d_n)$ and we are done.
\end{proof}

\begin{definition}
\label{30}
Let $C = (C^-,C^+)$ be a Dedekind cut in some real closed field $\cF$.

\noindent
1) $C$ is said to be positive if $C^- \cap \cF^+ \ne \phi$.

\noindent
2) A positive cut $C$ is said to be additive if $C^-$ is closed under
   addition.

\noindent
3) $C$ is said to be multiplicative if $C^- \cap F^+$ is closed under
   multiplication and $2 \in C^-$.

\noindent
4) $C$ is said to be a Scott cut if for every $t > 0$ in $\cF$ we have
   some $a \in C^-,b \in C^+$ with $b - a<t$.
\end{definition}

\begin{claim}
\label{31}
Let $\{C_1,C_2\}$ be a set of Dedekind cuts in some real closed field
$\cF$.  So the following are equivalent:
\mn
\begin{enumerate}
\item  $\{C_1,C_2\}$ is dependent
\sn
\item  for some polynomial $p(x,y,c)$ with $\bar c,d \subset \cF$ we have
that $\varphi(x,y,\bar c,d) = ``y$ is the smallest set to realize $y >
d$ and $p(x,y,\bar c) = 0"$ defines a strictly monotonic function $y
  = f(x)$ from some interval $I_1$ around $C_1$ onto some interval
  $I_2$ around $C_2$, such that the cut is respected, i.e. $f(I_1 \cap
  C_1) = I_2 \cap C^-_2$ or $f(I_1 \cap C^-_1) = I_2 \cap C^+_2$
\sn
\item  some $\cF$-definable function $\varphi(x,y,\bar c)$
  monotonically maps some interval around $C_1$ onto some interval
  around $C_2$ such that the cut is respected.
\end{enumerate}
\end{claim}

\begin{proof}
$(1) \Rightarrow (2)$  First assume 
$\{C_1,C_2\}$ is dependent.  By definition for some $a,a^*$ realizations
of $C_1,C_2$ respectively in some real closed field $\cF^+$ extending
$\cF$, and for some polynomial $p(x,y_1) \ne 0$ we have that
$\models p(a,a^*) = 0$.  Now let $b \in \cF^+$ be the smallest root
of $p(a,y)$ in $\cF^+$ that is larger than $C^-_2$ and let $d \in
C^-_2$ be some element smaller than $b$ and larger than all the
smaller roots of $p(a,y)$ in $\cF^+$.  Now take $\varphi(x,y) = ``y$ is
the smallest element to realize $y > d \wedge p(x,y) = 0"$.

Now surely $\models \varphi[a,b]$ and by Corollary \ref{8} we have that
$\varphi(x,y)$ defines $y$ as a function of $x$.  By quantifier
elimination the function is monotonic and the cut is respected as
$\models \varphi[a,b]$.

(3) follows from (2) trivially.

$(3) \Rightarrow (1)$  Now assume a function 
such as in (3) exists, and denote it $f(x)=y$.
Let $c_1$ be some realization of $C_1$ in some real closed field
extending $\cF$.  Now $\cF \models ``\varphi(x,y,\bar c)$ defines a
monotonic function on $(d^-,d^+)"$ where $d \in C_1,d^+ \in C^+_1$.
Hence $c_1$ has an image under $f(x)$, denoted $f(c_1)$.  Without
loss of generality the function is strictly increasing in
$(d^-,d^+)$.  So for every $a \in f(d^-,d^+) \cap C^-_2$ we have that
$f(x) > a$ on some end-segment of $C_1$ and so $f(c_1) > a$ as the
function is monotonic.  However, for every $b \in f(d^-,d^+) \cap
C_2$ by similar considerations we have that $f(c_1) < b$.  So $f(c_1)$
realize the cut $C_2$ and so every field extending $\cF$ realizing
$C_1$ also realize $C_2$ and so the cuts are dependent.
\end{proof}

\begin{definition}
\label{32}
Two Dedekind cuts $C_1,C_2$ in some real closed field $\cF$ will be
called equivalent if $\{C_1,C_2\}$ is dependent.  In this case we will
say that $C_i$ is equivalent to $C_{3-i}$ for $i=1,2$.  They are
positively equivalent if there exists some $\cF$-definable
order-preserving function from some interval about $C_1$ to some
interval about $C_2$ respecting the cuts and negatively equivalent if
there exists such an anti-order-preserving function.
\end{definition}

\begin{remark}
\label{33}
Every pair of equivalent cuts is either positively equivalent, negatively
equivalent or both.
\end{remark}

\begin{remark}
\label{34}
Positive equivalence is a transitive relation between Dedekind cuts
(the proof is immediate from the last claim).
\end{remark}

\noindent
The following claim strengthens the previous claim:
\begin{corollary}
\label{35}
Let $\{C_1,C_2\}$ be a set of Dedekind cuts in some real closed field
$\cF$.  So $C_1$ is positively equivalent to $C_2$ iff there exists
some $\cF$-definable function monotonically mapping some end-segment
of $C^-_1$ onto some end-segment of $C^-_2$.
\end{corollary}

\begin{proof}
The first direction follows from the Claim \ref{31}.

For the second direction let $\varphi(x,y,\bar c)$ be such a function
and \wilog \, assume $\varphi(x,y,\bar c)$ is strictly increasing on
$D^-$ an end segment of $C^-_1$.  Now examine the formula $\psi(x,\bar
c) \equiv ``\varphi(x,y,\bar c)$ defines $y$ as a strictly increasing
function of $x$ on some interval $x-\varepsilon,x +\varepsilon)"$.  By
quantifier elimination $\{x \in \cF :\models \psi(x,\bar c)\}$ is a
finite union of intervals and is non-empty on some end-segment of $C^-$,
hence is non-empty on some $(d^-,d^+)$ with $d^- \in C^-,d^+ \in C^+$.. Now
use the previous claim.
\end{proof}

\noindent
In dependent theories, the average type (see Definition \ref{11}) was
defined for indiscernible sequences.  The reason is that for general
sequences there may exist a formula $\varphi(x,\bar c)$ such that its
truth value shifts back and forth and the average will not be defined,
resulting in incomplete types.  In the theory of real closed fields,
by quantifier elimination, it is enough for the sequence to be endless
strictly increasing (or decreasing) for the average type to be complete.
\begin{definition}
\label{36}
Let $\cF$ be some real closed field and let $\bar{\bold c} = \langle
c_t:t \in I\rangle$ be some endless increasing sequence in $\cF$.  For
every $\Delta \subseteq \bbL(T)$ we define the $\Delta$-average type
of $\bar{\bold c}$ over $A \subset C$ to be:

\begin{equation*}
\begin{array}{clcr}
\Av_\Delta(A,\bar{\bold c}) = \{\varphi(x,\bar d)^{\bold t}:&\bar d
\subseteq A,\varphi(x,\bar Y) \in \Delta, \\
 &\text{ for every large enough } t \in I \text{ we have } \models
\varphi(c_t,\bar d)^{\bold t}\}
\end{array}
\end{equation*}

\mn
omitting the $\Delta$ means for $\Delta = \bbL(T)$.
\end{definition}

\begin{corollary}
\label{37}
Let $\cF$ be some real closed field and let $\bar{\bold c} = \langle
c_t:t \in I\rangle$ be some endless increasing sequence in $\cF$.  So
for every $A \subset \cC$ we have that $\Av(\bar{\bold c},A)$ is a
complete type over $A$.
\end{corollary}

\begin{proof}
Let $\varphi(x,\bar y) \in \bbL$ be some formula and $\bar d \in A$.
By quantifier elimination we have that $\{x \in \cF : \models
\varphi(x,\bar d)\}$ is a finite union of intervals in $\cF$.  Hence
for some truth value $\bold t$ for some end-segment $J$ of $I$ we have
that $\forall t \in J(\models \varphi(c_t,\bar d)^{\bold t})$.  So we
are done.
\end{proof}

\noindent
We leave the proof of the following claim to the reader.
\begin{claim}
\label{38}
Let $C$ be a Dedekind cut in some real closed field $\cF$ and let
$\bar{\bold c} = \langle c_t:t \in \alpha\rangle$ be some strictly
increasing unbounded sequence in $C^-$.  So $a \in \cC$ realizes $\cC$
in some real closed field extending $\cF$ iff $a$ realizes 
$\Av(\cF,\bar{\bold c})$.
\end{claim}

\begin{claim}
\label{39}
Let $\{C_1,C_2\}$ be a set of Dedekind cuts in some real closed field
$\cF$.  Assume that for some polynomial $p(x,y,\bar c) \ne 0$ with
$\bar c \in \cF$ we have that there exists some unbounded increasing
sequence $\bar{\bold a} = \langle a_t:t \in I\rangle$ in $C^-_1$ and
some unbounded increasing sequence $\bar{\bold b} = \langle b_t:t \in
I\rangle$ in $C^-_2$ such that $p(a_t,b_t,\bar c) = 0$ for every $t
\in I$.  Then $C_1,C_2$ are positively equivalent.
\end{claim}

\begin{proof}
Without loss of generality we can assume that there exists no $a \in
C^-_1$ such that $\forall b(p(a,b,\bar c) =0)$ (there can be only a
finite number of such $a$'s).  Let $a$ realize $C_1$ in some real
closed field extending $\cF$.  So by quantifier elimination, for some
end-segment $D_2$ of $C^-_2$ and $n \in \omega$ we have that $(\forall
z \in D_2)(p(a,y,\bar c) = 0$ has exactly $n$ solutions bigger than
$z$).  Now for every $b \in C^-_2$ for every large enough $s \in I$ we
have $b_s > b$, so for every large enough $t \in I$ we have $\exists
y(p(a_t,y,\bar c) = 0,y > b) \in \Av(\cF,\bar{\bold a})$.  So $n >
0$.  \Wilog \, for some $e \in D_2$ we have $\forall s \in I(b_s >
e)$.  Again \wilog \, we may assume that for every $t \in I$ there are
exactly $n$ solutions to $p(a_t,y,\bar c) = 0$ bigger than $e$ and by
quantifier elimination we have that for some end-segment $D_1$ of
$C^-_1$ we have that $\forall x \in D_1(p(x,y,\bar c)=0$ has exactly
$n$ solutions bigger than $e$).

Let $\varphi(x,y,\bar c) = ``y$ is the smallest solution to
$p(x,y,\bar c) = 0$ bigger than $e$.  So $\varphi(x,y,\bar c)$ defines
$y$ as a function of $x$ for $x \in D_1$.  \Wilog \, $\varphi(x,y,\bar
c)$ is strictly monotonic on $D_1$ and $\bar{\bold a} \subset D_1$.
Assume by contradiction that $\varphi(x,y,\bar c)$ is strictly
decreasing and let $d \in D_1$ and $\models \varphi[d,b,\bar c]$.  So
for every $x \in D_1$ bigger than $d$ we have that $p(x,y,\bar c)=0$
has a solution in $(e,b)$.  Hence for every large enough $t \in I$ we
have that $p(a_t,y,\bar c)=0$ has at most $n-1$ solutions bigger than
$b$.  So $p(a,y,\bar c)=0$ has at most $n-1$ solutions bigger than
$b$, contradicting the choice of $D_2$.  So $\varphi(x,y,\bar c)$ is
strictly increasing on $D_1$.

Let $C'_2$ be the downward closure in $\cF$ of the image of $D_1$
under the function $\varphi$, so $C'_2$ is an initial segment with no endpoint.
If we show that $C'_2 = C^-_2$ then we are done by Claim \ref{35}.  We
first show that $C'_2 \subseteq C^-_2$.  For every $g \in D_1$ there
exists some $t \in I$ with $a_t > g$.  By assumption $p(a_t,b_t,\bar
c)=0$ so the image of $a_t$ under $\varphi(x,y,\bar c)$ is in $C^-_2$.
Since $\varphi$ is monotonically increasing then the image of $g$ is
in $C^-_2$, too.  So the image of $D_1$ is a subset of $C^-_2$ hence
so is $C'_2$.  Now assume that $C'_2 \ne C^-_2$.  So for some $b \in
C^-_2$ we have $b > C'_2$.   So for every large enough $x \in C^-_1$
we have that $p(x,y,\bar c) = 0$ has a solution in $(e,b)$ and as
before we get a contradiction to the choice of $D_2$.
\end{proof}
\newpage

\section {Perpendicularity in real closed fields}

In this section we explore the meaning of perpendicular sequences in real
closed fields.  We prove that every cut with large enough cofinality
is induced by an indiscernible sequence.  We show that independence of
cuts in the sense of Definition \ref{28}
is equivalent to perpendicularity of their sequences (see Theorem
\ref{42} and Claim \ref{46}).

\begin{definition}
\label{40}
Let $\bar{\bold a} = \langle a_t:t \in I\rangle$ be some strictly
increasing indiscernible sequence.  The cut 
induced by $\bar{\bold a}$ in $\cF$ is the cut
$(C^-,C^+)$ with $C^-$ the downwards closure of $\bar{\bold a}$ in $\cF$.
\end{definition}

\begin{claim}
\label{41}
Let $(C^-,C^+)$ be a Dedekind cut in some real closed field $\cF$ with
left cofinality $\lambda > \aleph_0$.  Then there exists some
indiscernible sequence $\bar{\bold a} \subset \cF$ with
$\cf(\bar{\bold a}) = \lambda$ such that $C$ is induced by $\bar{\bold a}$.
\end{claim}

\begin{proof}
Let $\langle d_t:t \in \lambda\rangle$ be some cofinal sequence in
$C^-$.  Let $D$ be some ultrafilter on $C^-$ extending the family of
end sections of $C^-$.  We now find an endless indiscernible sequence
$\bar{\bold b} \subseteq \cC$ of order type $\omega$ based on $D$ over
$\cF$.  We now define by induction $a_t \in C^-$ for $t \in \lambda$
such that for every $t \in \lambda,a_t$ realizes $\Av(\bar{\bold b}
\cup \{a_s:s < t\},\bar{\bold b})$ and $d_t < a_t$.  If we succeed
then $\bar{\bold a} = \langle a_t:t \in \lambda\rangle$ is as required
(it is indiscernible as a sequence based on $\bar{\bold b}$).  At the
$t$-th stage, for every $\varphi(x) \in \Av(\bar{\bold b} \cup \{a_s:s
< t\},\bar{\bold b})$ there exists by definition of average some $b_s$
with $\models \varphi(b_s)$ and $s$ is bigger than all the indices of
the $b_s$'s appearing as parameters in $\varphi(x)$.  Recall that
$b_s$ realizes $\Av(C^- \cup \{b_\ell:\ell < s\},D)$ and by the choice
of $D$ there exists some unbounded $A_\varphi \subset C^-$ with each
$a \in A_\varphi$ realizing $\varphi(x),a > d_t$.  

By quantifier elimination of the theory of real closed fields, for
every model $\cT$ and formula $\chi(x)$ with parameters from $\cT$ we
have that $\{\chi(x):x \in \cT\}$ is a finite union of convex subsets
of $\cT$.  So \wilog \, we choose $A_\varphi$ such that $A_\varphi$ is
an end-section of $C^-$.  Now $|\{\bar{\bold b} \cup \{a_s:s < t\}| =
\aleph_0 < \cf(C^-)$ so there exists some $A_t \subset C^-$ an
end-section with each $a \in A_t$ realizing $\Av(\bar{\bold b} \cup
\{a_s:s < t\},\bar{\bold b})$.  Take some $a_t \in A_t$.  So the
induction is successful and we are done. 
\end{proof}

\begin{question}
Is this true without assuming $\lambda > \aleph_0$?
\end{question}

\begin{answer}
No.  Consider some $\aleph_1$-saturated real-closed field $\cS$.  Take
$C^-$ the downward closure of $\bbN \subset \cS$.  Surely it's a
Dedekind cut, and no indiscernible sequence induce it.  However, if we
restrict to finite $\Delta \subseteq \bbL(T)$ there exists a
$\Delta$-indiscernible sequence inducing the cut.  See the original proof.
\end{answer}

\begin{theorem}
\label{42}
Let $\cF$ be a real closed field, $C_1,C_2$ Dedekind cuts in $\cF$
induced by the indiscernible sequences $\bar{\bold a} = \langle a_t:t
\in I\rangle$ and $\bar{\bold b} = \langle b_s:s \in J\rangle$
respectively.  Then $C_2$ is positively equivalent to $C_1$ iff
$\bar{\bold a},\bar{\bold b}$ are not $\cF$-perpendicular in the sense
of Definition \ref{22}.
\end{theorem}

\begin{proof}
First assume that $C_1,C_2$ are positively dependent.  So there exist
intervals $I_1$ around $C_1$ and $I_2$ around $C_2$ and a polynomial
$p(x,y,\bar c)$ with $\bar c \subset \cF$ such that $p(x,y,\bar c) =
0$ defines a function from $I_1$ onto $I_2$ respecting the cuts,
denoted $y = f(x)$.   Define $\varphi(x,y,\bar c) = ``f(x)=z$ and $y
\ge z"$.  Now for every large enough $t \in I$ we have $f(a_t) \in
C_2$.  Hence for some large enough $s \in J$ we have $b_s > f(a_t)$
and so $\models \varphi(a_t,b_s,\bar c)$.  On the other hand for every
large enough $s \in J$ for every large enough $t \in I$ we have
$f(a_t) > b_s$ and so $\models \neg \varphi(a_t,b_s,\bar c)$.  So by
definition $\bar{\bold a},\bar{\bold b}$ are not $\cF$-perpendicular.

For the second direction, we will use Claim \ref{39}.  So $\bar{\bold
  a},\bar{\bold b}$ are not $\cF$-perpendicular.  By quantifier
  elimination for some polynomial $p(x,y,\bar c)$ we have that
\mn
\begin{enumerate}
\item  for every large enough $t \in I$ for every large enough $s \in
  J$ we have $p(a_t,b_s,\bar c) > 0$
\sn
\item  for every large enough $s \in J$ for every large enough $t \in I$
 we have $p(a_t,b_s,\bar c) > 0$.
\end{enumerate}
\mn
We shall now find an ordinal $\lambda$ and $\bar{\bold a} = \langle
t_\alpha:\alpha < \lambda \rangle,\bar{\bold b}' = \langle
s_\alpha:\alpha < \lambda\rangle$ unbounded sequences of $I$ with the
property that for every $\alpha,\beta \in \lambda$ we have $\models
p(a_{t_\alpha},b_{s_b},\bar c) > 0$ iff $\beta \ge \alpha$.  So let
$\lambda = \cf(I)$ and $\bar{\bold c} = \langle c_\alpha:\alpha <
\lambda\rangle$ be an unbounded sequence in $I$.  We will define
$\langle t_\alpha:\alpha < \lambda\rangle,\langle s_\alpha:\alpha <
\lambda\rangle$ by induction on $\alpha < \lambda$ such that $t_\alpha
> c_\alpha$ for every $\alpha < \lambda$.  For $\alpha = 0$ let $t_0$
be large enough as in (1) and larger than $c_0$.  Let $s_0$ be large
enough as in (1) for $t_0$ and large enough as in (2).  For limit
ordinals $\alpha < \lambda$ we know that for every $\beta < \alpha$
for $b_{s_b}$ we have some large enough $t_{s_b}$ such as in (2) and
since $a < \cf(I)$ some $t_\alpha$ is large enough as in (2) for every
$\beta < \alpha$ and larger than $\langle t_\beta:\beta <
\alpha\rangle$ and larger than $c_\alpha$.  Let $s_\alpha$ be large
enough as in (1) for $t_\alpha$.  For $\alpha = \beta +1$ let
$t_\alpha$ be large enough as in (2) for $b_{s_\beta}$ and larger than
$c_\alpha$.  Let $s_\alpha$ be large enough as in (1) for $t_\alpha$.
One can easily check that the condition that for every $\alpha,\beta
\in \lambda$ we have $\models p(a_{t_\alpha},b_{s_b},\bar c) > 0$ iff
$\beta \ge \alpha,\bar{\bold a}'$ is unbounded since $t_\alpha >
c_\alpha$ for every $\alpha < \lambda$ and $\bar{\bold c} = \langle
c_\alpha:\alpha < \lambda\rangle$ is unbounded in $I$.  We will now
prove that $\bar{\bold b}'$ is unbounded in $J$.  Assume otherwise,
then for some large enough $s \in UJ$ we have that $s$ is large enough
as in (1) for every $\langle t_\alpha:\alpha < \lambda\rangle$ and $s$
is larger than $\langle s_\alpha:\alpha < \lambda\rangle$.

Now take $t \in I$ large enough as in (2) for $s$, so $p(a_t,b_s,\bar
c) < 0$.  We know that $t < t_\alpha$ for some $\alpha < \lambda$
since $\bar{\bold a}'$ is unbounded, so $p(a_{t_\alpha},b_s,\bar c) <
0$.  Now $s_\alpha$ is large enough as in (1) for $t_\alpha$ and $s >
s_\alpha$ so $p(a_{t_\alpha},b_s,\bar c) > 0$-contradiction.

Now for every $\alpha < \lambda$ we have that
$p(a_{t_\alpha},b_{s_\alpha},\bar c) > 0$ and 
$p(a_{t_{\alpha +1}},b_{s_\alpha},\bar c) <0$ so for some $d_\alpha
\in (a_{t_\alpha},a_{t_{\alpha + a}})$ we have
$p(d_\alpha,b_{s_\alpha},\bar c) = 0$.  Now by Claim \ref{39} we are done.
\end{proof}

\begin{claim}
\label{43}
Let $\bar{\bold a} = \langle a_t:t \in I\rangle$ be an endless
strictly increasing sequence in some real closed field $\cF$.  Then
for some $\cS$ a real closed field extending $\cF$ we have that
$\bar{\bold a}$ induces a Dedekind cut in $\cS$ and in every real
closed field extending $\cS$.
\end{claim}

\begin{proof}
Take $\cS$ to be a real closed field extending $\cF$ with a realization
of the type $\{x > 0\} \cup \{x < a_t - a_s:t,s \in I,t > s\}$ and of
the type $\{x > a_t:t \in I\}$.  The rest is left to the reader.
\end{proof}

\begin{corollary}
\label{44}
In the theory of real closed fields, not being perpendicular is a
transitive relation on endless indiscernible sequences of elements.
\end{corollary}

\begin{proof}
Let $\bar{\bold a},\bar{\bold b},\bar{\bold c}$ be endless
indiscernible sequences of elements in some real closed field $\cF$
such that $\bar{\bold a},\bar{\bold b}$ are not perpendicular and
$\bar{\bold b},\bar{\bold c}$ are not perpendicular.  So for some real
closed $\cS$ we have that $\bar{\bold a},\bar{\bold b}$ are not
$\cS$-perpendicular, $\bar{\bold b},\bar{\bold c}$ are not
$\cS$-perpendicular and that $\bar{\bold a},\bar{\bold b},\bar{\bold
  c}$ induce Dedekind cuts in $\cS$, denoted $\cC_a,\cC_b,\cC_c$.  So
by the previous thoerem $\cC_a,\cC_b$ are positively equivalent, and
$\cC_b,\cC_b$ are positively equivalent and by transitivity
$\cC_a,\cC_c$ are positively equivalent.  So $\bar{\bold a},\bar{\bold
  c}$ are not $\cS$-perpendicular and so are not perpendicular. 
\end{proof}

\begin{example}
\label{45}
Take $T = \Th(\bbR^2,<_1,<_2)$ (so $T$ is dependent because it can be
interpreted in $\bbR$ and $\bbR$ is dependent) and

\[
\bar{\bold a} = \langle (n,0):n \in \bbN\rangle,\bar{\bold b} =
\langle (n,n):n \in \bbN\rangle,\bar{\bold c} = \langle (0,n):n \in
\bbN\rangle.
\]
\end{example}

\begin{remark}
In the next Claim we replace $\cF$-perpendicular Definition \ref{22},
in Theorem \ref{42} by prependicular, Definition \ref{15} is a special case.
\end{remark}

\begin{claim}
\label{46}
Let $\cC_1,\cC_2$ be two Dedekind cuts in some real closed field $\cF$
such that
the left cofinality of $\cC_\ell$ is strictly smaller than the right
cofinality of $\cC_\ell$ and larger than $\aleph_0$ for $\ell \in
\{1,2\}$.  Let $\bar{\bold a}^1 = \langle a^1_t:t \in
I^1\rangle,\bar{\bold a}^2 = \langle a^2_t:t \in I^2\rangle$ induce
$\cC_1,\cC_2$ respectively.  So $\cC_1$ is positively equivalent to
$\cC_2$ iff $\bar{\bold a}^1,\bar{\bold a}^2$ are not perpendicular. 
\end{claim}

\begin{proof}
First assume that $\cC_1$ is positively equivalent to $\cC_2$.  So by
Theorem \ref{42} we have that $\bar{\bold a}^1,\bar{\bold a}^2$ are
not $\cF$-perpendicular and so are not perpendicular.

Now assume $\bar{\bold a}^1,\bar{\bold a}^2$ are not perpendicular and
for a contradiction assume $\cC_1$ is not positively equivalent to
$\cC_2$.  First, by Theorem \ref{25} we have that $\cf(\bar{\bold
  a}^1) = \cf(\bar{\bold a}^2)$ and by Theorem \ref{42} we conclude
that $\bar{\bold a}^1,\bar{\bold a}^2$ are $\cF$-perpendicular.  We
will use Claim \ref{24} to show that $\bar{\bold a}^1,\bar{\bold a}^2$
are perpendicular.  We will choose by induction on $n \in \omega$ two
elements $b^1_n,b^2_n$ of $\cF$ such that the conditions in the claim
occur.

So in the $n$-th stage, we first find $b^1_n \in \cF$ which realize

\[
q(x) = \Av(\bar{\bold a}^1 \cup \bar{\bold a}^2 \cup \{b^k_m:k \in
\{1,2\},m < n\},\bar{\bold a}^\ell).
\]

\mn
So for every $\varphi(x) \in q(x)$ for every large enough $t \in I^1$
we have $\models \varphi[a_t]$.  By quantifier elimination the set
$\{x \in \cF: \models \varphi[x]\}$ is a finite union of
intervals. So for some $b_\varphi \in \cC^+_1$ we have that for every
$c \in \cC^+_1$ such that $e < b_\varphi$ we have $\models
\varphi[e]$.  Since $|q(x)| = \cf(\bar{\bold a}^1)$ and the latter is
smaller than the right cofinality of $\cC_1$ then for some $b \in
\cC^+_1$ we have that $b < b_\varphi$ for every $\varphi \in q$.  In
other words $b$ realizes $q$.  Now just denote $b^1_n = b$ and we are
done.  $b^2_n$ is chosen in the same manner.

This contradiction concludes the proof. 
\end{proof}
\newpage

\section {Strong perpendecularity in real closed fields}

In this item we explore strong perpendicularity of sequences in real
closed fields.  We first define neighbor sequences and strong
perpendicularity.  We then show that strong perpendicularity is
invariant to reversing the order of a sequence, applying a definable
function and to taking pre-images of definable functions.  The section
ends with the result that there are no strongly-perpendicular sequences
in real closed fields (Theorem \ref{63}).

\begin{definition}
\label{47}
1) Two indiscernible sequences $\langle a^\ell_t:t \in I^\ell,\ell \in
   \{1,2\}\rangle$ will be called immediate neighbors (or inb's for
   short) if there exists some indiscernible sequence $\langle b_s:s
   \in J\rangle$ and order preserving or anti-order preserving
   injections $\sigma^\ell:I^\ell \rightarrow J$ such that for $\ell
   \in \{1,2\},t \in I^\ell$ we have $a^\ell_t = b_{\sigma^\ell(t)}$.
   In other words, two sequences are immediate neighbors if they can
   be embedded in the same indiscernible sequence.

\noindent
2) Two indiscernible sequences $\bar{\bold a}^1,\bar{\bold a}^2$ will
   be called neighbors (or nb's for short) if there is a finite
   sequence $\bar{\bold b}^0,\bar{\bold b}^1,\dotsc,\bar{\bold
   b}^\ell$ of indiscernible sequences such that $\bar{\bold a}^1 =
   \bar{\bold b}^0,\bar{\bold a}^2 = \bar{\bold b}^\ell$ and
   $\bar{\bold b}^i,\bar{\bold b}^{i+1}$ are inb's for
   $i=0,\dotsc,\ell-1$.  In this case we say that the sequences are
   $\ell$-nb's. Note that to say that two sequences are 1-nb's is the
   same as saying that they are inb's.
\end{definition}

\noindent
Recall Definition \ref{30}.
\begin{example}
\label{48}
Let $\bar{\bold a}^1,\bar{\bold a}^2$ be two indiscernible sequences
in some real closed field $\cF$ such that both induce multiplicative
cuts.  Then $\bar{\bold a}^1,\bar{\bold a}^2$ are 2-nb's.
\end{example}

\begin{proof}
The proof is based on the claim below.

We construct a third sequence $\bar{\bold b} = \langle b_s:s \in
\omega\rangle$ such that $\bar{\bold a}^1,\bar{\bold b},\bar{\bold
  a}^2,\bar{\bold b}$ are both indiscernible.  This will prove the
claim with the sequence $\bar{\bold a}^1,\bar{\bold b},\bar{\bold
  a}^2$ witnessing that $\bar{\bold a}^1,\bar{\bold a}^2$ are indeed
2-nb's.  So we begin with some endless well-ordered set $\omega$ and
construct $b_s$ by induction on $s \in \omega$.  Assume $b_t$ has
already been defined for every $t < s$.  We define $b_s$ to be some
element from $\cC$ realizing the type

\begin{equation*}
\begin{array}{clcr}
\{x>p(a^1_{q_0},\dotsc,a^1_{q_m},a^2_{r_0},\dotsc,a^2_{r_n},b_{t_0},
\dotsc,b_{t_k}):&p \text{ is some polynomial}, \\
  &a^1_{q_0},\dotsc,a^1_{q_m} \text{ is a subsequence of } \bar{\bold
  a}^1, \\
 &a^2_{r_0},\dotsc,a^2_{r_n} \text{ is a subsequence of } \bar{\bold
  a}^2, \\
  &t_0 <_J \ldots <_J t_k <_J s\}.
\end{array}
\end{equation*}

\mn
Now use the claim below to show that $\bar{\bold a}^1,\bar{\bold
  b},\bar{\bold a}^2,\bar{\bold b}$ are both indiscernible.
\end{proof}

\begin{claim}
\label{49}
Let $\bar{\bold a} = \langle a_t:t \in I\rangle$ be some increasing
sequence of positive elements in some real closed field $\cF$.  Then
$\bar{\bold a}$ is indiscernible and induces a multiplicative cut in
$\cF$ iff for every $k \in \bbN,t_0 <_I \ldots <_I t_{k-1} <_I t_k$ we
have that $\models \{a_{t_k} >_{\cF}
P(a_{t_0},\dotsc,a_{t_{k-1}}):P(x_0,\dotsc,x_{k-1})$ is some
parameter-free polynomial$\}$.
\end{claim}

\begin{proof}
First assume that $\bar{\bold a}$ is indiscernible and induce a
multiplicative cut in $\cF$.  Let $k \in \bbN,P(x_0,\dotsc,x_{k-1})$
some polynomial and $t_0 <I \ldots <_I t_k$ induces from $I$ and we
will prove that $\models a_{t_\alpha} >_{\cF}
P(a_{t_0},\dotsc,a_{t_{k-1}})$.  We know that $\bar{\bold a}$ is
indiscernible so it is enough to prove the above for some increasing
sub-sequence in $I$ of length $k+1$.  Now $P(x_0,\dotsc,x_{k-1}) =
\sum\limits^{m}_{i=0} q_i(x_0,\dotsc,x_{k-1})$ where $q_i$ are
monomial.  $\bar{\bold a}$ induces a multiplicative cut so for every
$i \in \{0,\dotsc,m\}$ there exists some $s_i \in I,s_i >_I t_{k-1}$
such that $a_{s_i} > q_i(a_{t_0},\dotsc,a_{t_{k-1}})$.  Since
$\bar{\bold a}$ also induces an additive cut (the reader can easily
check that every multiplicative cut is also an additive cut) for some $s
\in I,s >_I t_{k-1}$ we have that $a_s > \sum\limits^{m}_{i=0}
a_{s_i}$.  All together we have $a_s > \sum\limits^{m}_{i=0} q_i
(a_{t_0},\dotsc,a_{t_{k-1}}) = P(a_{t_0},\dotsc,a_{t_{k-1}})$ and we
  are done.

For the second direction we first prove that $\bar{\bold a}$ is
indiscernible.  Denote by $\bbP$ the set of all polynomials.  By
assumption we have that $\bar{\bold a}$ is
$(\bbP,\phi)$-indiscernible.  By quantifier elimination for the theory
of real closed fields we have that $\bar{\bold a}$ is indiscernible.
To prove that $\bar{\bold a}$ induces a multiplicative cut we denote
by $(C^-,C^+)$ the cut induced by $\bar{\bold a}$ and assume $x,y \in
\cF$ such that $x,y \in C^- \cap \cF^+$.  By definition we have that
for some $t \in I:a_t > x,y$.  Let $s >_I t$.  So by assumption $a_s >
p(a_t)$ for every polynomial $p$, specifically $p(x) = x^2$.  So $a_s
> a^2_t > x \cdot y$ (remember that all the elements discussed here
are positive) and specifically $x \cdot y \in C^-$.  Hence $(C^-,C^+)$
is multiplicative and we are done.
\end{proof}

\begin{claim}
\label{50}
Let $\langle a_t:t \in I\rangle$ be some indiscernible sequence, and
denote by $I^*$ the set $I$ ordered in reversed order.  So $\langle
a_t:t \in I\rangle,\langle a_t:t \in I^* \rangle$ are immediate nb's.
\end{claim}

\begin{proof}
The identity function from $I$ to itself is enough - just look at the
definitions. 
\end{proof}

\begin{notation}
\label{51}
Let $\bar{\bold a}$ be some infinite indiscernible sequence.  We
denote by $\tp'(\bar{\bold a})$ the type of some infinite countable
subsequence of $\bar{\bold a}$, we call $\tp'(\bar{\bold a})$ the
local type of $\bar{\bold a}$.
\end{notation}

\begin{claim}
\label{52}
Let $\bar{\bold a}^1,\bar{\bold a}^2$ be some $\ell$-nb's
indiscernible sequences for some $\ell \in \bbN$ then $\tp'(\bar{\bold
  a}^1) = \tp'(\bar{\bold a}^2)$ or $\tp'(\bar{\bold a}^1) =
\tp'(\bar{\bold a}^{2,*})$.  In other words, the local type is
preserved or reversed between neighbors.
\end{claim}

\begin{proof}
Using induction on $\ell$ we are left with proving that the claim is
correct in the case of immediate neighbors.  This case is also easy
since if $\bar{\bold a}^i$ is embedded in $\bar{\bold b}$, then surely
$\tp'(\bar{\bold b}) = \tp'(\bar{\bold a}^i)$.    
\end{proof}

\noindent
We next turn our attention to a different view of perpendicularity;
one which eliminates, in some sense, the dependency of the definition
of perpendicularity in the particular choice of the sequences
themselves.
\begin{definition}
\label{53}
Let $\bar{\bold a}^1,\bar{\bold a}^2$ be some endless indiscernible
sequences in some real closed field $\cF$.  We say that $\bar{\bold
  a}^1,\bar{\bold a}^2$ are strongly perpendicular iff $\bar{\bold
  c}^1,\bar{\bold c}^2$ are perpendicular whenever $\bar{\bold
  a}^i,\bar{\bold c}^i$ are nb's for $i \in \{1,2\}$. 
\end{definition}

\begin{example}
\label{54}
Consider the following model $M:|M| = \{(x,y):x,y \in \bbQ\},<_1$ is a
binary predicate agreeing in every point with the natural order of
$\bbQ$ according to the first element of the tuple, and $<_2$ is a
binary predicate agreeing in every point with the natural order of
$\bbQ$ according to the second element of the tuple.

(The reader should check that this model is indeed a model of a
dependent theory).  Indiscernible sequences $\langle (a_t,b_t):t \in
I\rangle$ in this model divide into 8 kinds, according to the
direction in each dimension (their local types):
\mn
\begin{enumerate}
\item  $\forall t <_I s(a_t = a_s,b_t <_{\bbQ} b_s)$
\sn
\item  $\forall t <_I s(a_t = a_s,b_t >_{\bbQ} b_s)$
\sn
\item  $\forall t <_I s(a_t >_{\bbQ} a_s,b_t = b_s)$
\sn
\item  $\forall t <_I s(a_t <_{\bbQ} a_s,b_t = b_s)$
\sn
\item  $\forall t <_I s(a_t <_{\bbQ} a_s,b_t <_{\bbQ} b_s)$
\sn
\item  $\forall t <_I s(a_t <_{\bbQ} a_s,b_t >_{\bbQ} b_s)$
\sn
\item  $\forall t <_I s(a_t >_{\bbQ} a_s,b_t <_{\bbQ} b_s)$
\sn
\item  $\forall t <_I s(a_t >_{\bbQ} a_s,b_t >_{\bbQ} b_s)$.
\end{enumerate}
\mn
In this example we will show that two indiscernible sequences of the
5th kind are always 2-nb's, that indiscernible sequences of the 1st
kind are strongly perpendicular to indiscernible sequences of the 3rd
kind and not strongly perpendicular to sequences of the 5th kind.

Start with two 5th kind indiscernible sequences: $\bar{\bold h}^1 =
\langle h^1_n:n \in \bbN\rangle,\bar{\bold h}^2 = \langle h^2_n:n \in
\bbN\rangle$.  We construct a third sequence $\bar{\bold h} = \langle
h_n:n \in \bbN\rangle$ such that $\bar{\bold h}^1,\bar{\bold
  h},\bar{\bold h}^2,\bar{\bold h}$ are both indiscernible.  This will
show that indeed the sequences are 2-nb's.  The construction is by
induction on $n \in \bbN$.  Assume $h_k$ has been chosen for all
$k<n$.  Choose $h_n$ such that $h_n >_j h^i_m$ for all $j \in
\{1,2\},i \in \{1,2\},m \in \bbN$ and $h_n >_j h_k$ for all $j \in
\{1,2\},k<n$.  The reader should now check that indeed what was
expected of $\bar{\bold h}$ occurs.

Now we prove that indiscernible sequences of the 1st kind are always
strongly perpendicular to indiscernible sequences of the 3rd kind.  We
notice that since local types are preserved between neighbors and the
division into kinds was based on the local type, it is enough to prove
that every sequence of the 1st kind is perpendicular to every sequence
of the 3rd kind.  We use the fact that $\{x >_1 y,x >_2 y\}$ is an
elimination set for this theory and only use this set when proving
perpendicularity.  So let $\bar{\bold a}^1 = \langle (x^1_t,y^1_t):t
\in I^1\rangle,\bar{\bold a}^2 = \langle (x^2_s,y^2_s:s \in
I^2)\rangle$ be two indiscernible sequences of the 1st and 3rd kind
respectively.  So $x^1_t$ are constant and equal to some $x^1_0$ and
$y^2_s$ are also constants and equal to some $y^2_0$.

Now let $\varphi_i(x,y) = x >_i y$ for $i \in \{1,2\}$.  Assume that
for every large enough $t \in I^1$ we have that for every large enough
$s \in I^2$ the following holds: $(x^1_t,y^1_t) >_1 (x^2_s,y^2_s)$.
So for every large enough $t \in I^1$ for every $s \in I^2$ we have
that $x^1_0 >_{\bbQ} x^2_s$ in particular for every large enough $s
\in I^2$ for every large enough $t \in I^1:(x^1_t,y^1_t) >_1
(x^2_s,y^2_s)$.  The opposite case as well as the cases involving
$>_2$ are handled the same way, and since this is an elimination set,
we are done.

Here we shall prove that sequences of the 1st kind are not strongly
perpendicular to sequences of the 5th kind.  Let $\bar{\bold a}^1 =
\langle (x^1_0,y^1_t):t \in I\rangle$ be a sequence of the 1st kind.
By definition in order to prove our claim it is enough to show that
$\bar{\bold a}^1$ is not perpendicular to some sequence of the 5th
kind which is a neighbor of a given sequence.  However, we showed
earlier in this example that any two sequences of the 5th kind are
nb's, so it is enough to show that $\bar{\bold a}^1$ is not
perpendicular to some sequence $\bar{\bold a}^2$ of the 5th kind.  So
let $\langle x^2_t:t \in I\rangle$ be some increasing sequence in
$\bbQ$.  We choose $\bar{\bold a}^2 = \langle (x^2_t,y^1_t):t \in
I\rangle$.  The formula $\varphi(x,y) = x >_2 y$ now witness that
$\bar{\bold a}^1,\bar{\bold a}^2$ are not perpendicular and we are done.
\end{example}

\begin{definition}
\label{55}
Let $I$ be some well-ordered set and $k \in \omega$.  We say that $I'
\subseteq I$ is $k$-spaced in $I$ if $I'$ is well-ordered and $t +_I k
<_I t'$ whenever $t < t'$ are both in $I'$.
\end{definition}

\begin{example}
\label{56}
$5-\bbZ$ is 4-spaced in $\bbZ$.
\end{example}

\begin{claim}
\label{57}
Let $\bar{\bold a} = \langle a_t:t \in I\rangle$ ($I$ well ordered) be
some endless indiscernible sequence over $A$ in some monster model
$\cC$ and let $\varphi(x_0,\dotsc,x_{n-1},y)$ be some formula with
parameters from $A$ which defines a function $f(\bar x) =
f_\varphi(\bar x)$ from $(\cup \bar{\bold a})^n$ to $\cC$.  So any
sequence of the form $\langle f(a_t,a_{t+1},\dotsc,a_{t+n-1}):t \in
I'\rangle$ where $I' \subseteq I$ is $n$-spaced is also indiscernible
over $A$.  (Put otherwise, if we are givenan $n$-ary definable
function $\varphi$ and a partition of $I$ into a sequence $J$ of
consecutive $n$-tuples then the image of $J$ under $\varphi$ remains
indiscernible.) 
\end{claim}

\begin{proof}
Denote $\bar b_t = \langle a_t,\dotsc,a_{t+n-1}\rangle$ so
$f(\bar{\bold a}) = \langle f(\bar b_t):t \in I'\rangle$.  Assume
otherwise, so for some $\varphi(x_0,\dotsc,x_{n-1})$ with parameters
from $A$ we have that for some $t_0 < \ldots < t_{n-1}$ and $s_0 <
\ldots < s_{n-1}$ we have that

\[
\varphi[f(\bar b_{t_0}),\dotsc,f(\bar b_{t_{n-1}})] \Leftrightarrow
\neg \varphi[f(\bar b_{s_0}),\dotsc,f(\bar b_{s_{n-1}})].
\]

\mn
But the formula $f(\varphi)(x_0,\dotsc,x_{n-1}) =
\varphi(f(x_0),\dotsc,f(x_{n-1}))$ is a formula over $A$ so

\[
f(\varphi)[\bar b_{t_0},\dotsc,\bar b_{t_{n-1}}] \Leftrightarrow
f(\varphi)[\bar b_{s_0},\dotsc,\bar b_{s_{n-1}}].
\]

\mn
And we have a contradiction.
\end{proof}

\begin{notation}
\label{58}
We denote $f_{I'}(\bar{\bold a}) = \langle f(a_t,\dotsc,a_{t+n-1}):t
\in I'\rangle$.
\end{notation}

\begin{claim}
\label{59}
Let $\bar{\bold a} = \langle a_t:t \in I\rangle$ ($I$ is well ordered)
be some infinite indiscernible sequence over $A$ in some monster model
$\cC$ and let $\varphi(\bar x,y)$ be some formula with parameters from
$A$ which defines a non-constant function $f(\bar x) = f_\varphi(\bar
x)$ on $(\cup \bar{\bold a})^n$.

\noindent
1) Let $\bar{\bold c}$ be some indiscernible sequence which is a nb's
of some sequence $f_{I'}(\bar{\bold a})$ where $I' \subseteq I$ is
   $n$-spaced.  Then for some indiscernible sequence $\bar{\bold a}' =
   \langle a'_t:t \in J\rangle$, which is a nb's of $\bar{\bold a}$,
we have that $\bar{\bold c} = f_{J'}(\bar{\bold a}')$ for some
   unbounded $n$-spaced $J' \subseteq J$.

\noindent
2) Assume that $\bar{\bold a}$ is endless and for some indiscernible
   endless sequence $\bar{\bold c}$ we have that $\bar{\bold c}$ is
   not perpendicular to some $f_{I'}(\bar{\bold a})$ where $I'
   \subseteq I$ is $n$-spaced and unbounded.  Then $\bar{\bold a}$ is
   not perpendicular to $\bar{\bold c}$.
\end{claim}

\begin{proof}
1) It is enough to prove the claim for the case when $\bar{\bold c}$
and $f_{I'(\bar{\bold a})}$ are inb's.  The general case is easily
   concluded by induction on the length of the sequence of i-nb's
   between $\bar{\bold c}$ and $f(\bar{\bold a})$.  So let $\bar{\bold
   b}$ be an indiscernible sequence such that $\sigma$ is an
   order-preserving injection of $\bar{\bold c}$ into $\bar{\bold b}$
   and $\tau$ is an order-preserving injection of $f_{I'}(\bar{\bold
   a})$ in $\bar{\bold b}$ (the case where one or both injections are
   anti-order-preserving is proven in the same manner).  We will find
   an indiscernible sequence $\bar{\bold a}_*$ with indices set $I_*$
   such that $\bar{\bold a} \subseteq \bar{\bold a}_*$ and
   $f_{I'}(\bar{\bold a}_*) = \bar{\bold b}$ for some unbounded
   $n$-spaced $I'_* \subseteq I_*$.  This is easily done using
   compactness and the fact that $\bar{\bold a},\bar{\bold b}$ are
   indiscernible and infinite.  Now all we have to do is take the
   subsequence of $\bar{\bold a}_*$ corresponding to the image of
   $\bar{\bold c}$ in $\bar{\bold b}$ under $\tau$. 

\noindent
2) $\bar{\bold c} = \langle c_s:s \in J\rangle$ is not perpendicular
   to $f_{I'}(\bar{\bold a}) = \langle f(a_t,\dotsc,a_{t+n-1}):t \in
   I'\rangle$, hence some $\psi(x,y)$ witness it, i.e. for every large
   enough $t \in I'$ for every large enough $s \in J$ we have that
   $\models \psi[c_s,f(a_t,\dotsc,a_{t+n-1})]$ and for every large
enough $s \in J$ for every large enough $t \in I'$ we have that
   $\models \neg \psi[c_s,f(a_t,\dotsc,a_{t+n-1})]$.  Now the formula
   $\psi_f(x,\bar y) = \psi(x,f(\bar y))$ witness the fact that
   $\bar{\bold c},\bar{\bold a}$ are not perpendicular in the same manner.
\end{proof}

\noindent
Recall the convention on $*$ in Claim \ref{50}.
\begin{claim}
\label{60}
Let $\bar{\bold a} = \langle a_t:t \in I\rangle$ be some positive
indiscernible sequence endless in both ways inducing an additive cut
in some real closed field $\cF$.  Then for some $\varphi(x,y)$ which
defines a function $f$ on $(\cup \bar{\bold a})^k$ we have that either
$f_{I'}(\bar{\bold a})$ or $f_{I'}(\bar{\bold a}^*)$ induces a
multiplicative cut in $\cF$ for some unbounded $k$-spaced $I'
\subseteq I$.
\end{claim}

\begin{proof}
Assume for a contradiction that for every $t_3 >_I t_2 >_I t_1 >_I t_0$
we have that:

\[
\frac{a_{t_1}}{a_{t_0}} = \frac{a_{t_3}}{a_{t_2}}
\]

\mn
Then by indiscernibility we have that the sequence is constant,
contradiction.  Without loss of generality we can assume that
$\bar{\bold a}$ is increasing, otherwise use $\bar{\bold a}^*$
instead.  At first, we assume the following:

\[
\frac{a_{t_1}}{a_{t_0}} < \frac{a_{t_3}}{a_{t_2}} \text{ whenever }
t_3 >_I t_2 >_I t_1 >_I t_0
\]

\mn
Now let $I'$ be 2-spaced in $I$ and consider $f(x,y) = \frac yx$.  We
want to prove that $f_{I'}(\bar{\bold a})$ induce a multiplicative
cut, i.e. that it is closed under multiplication and contains 2.  
So let $\frac{a_{t+1}}{a_t},\frac{a_{s+1}}{a_s}$ be two elements
of $f_{I'}(\bar{\bold a})$ with $s >_I t$.  (The indices $s+1,t+1$ are
not well-defined and we use them here for convenience. Formally we
mean that the index in the proof is in the interval $(s,s+1)$ or
$(t,t+1))$.  By equivalence (1) and the fact that $\bar{\bold a}$ is
indiscernible we have that for some $\ell >_I s$ we have that
$\frac{a_{\ell+1}}{a_\ell} > \frac{a_{s+1}}{a_\ell} >
\frac{a_{s+1} a_{t+1}}{a_s a_t}$.  Now since $\bar{\bold a}$ is
additive and strictly increasing
we have that $\frac{a_{t+1}}{a_t} > 2$ always.  If the reverse
inequality in (1) holds, we consider $f(x,y) = \frac xy$ instead and in
$\bar{\bold a}^*$.  Let $I'$ be 2-spaced in $I^*$.  Let
$\frac{a_{t+1}}{a_t},\frac{a_{s+1}}{a_s}$ be two elements in
$f_{I'}(\bar{\bold a})$ with $s >_I t$ (remember that $s <_{I'} t$).
By the inverse of (1) there exists $\ell >_{I^*} t$ such that
$\frac{a_{\ell+1}}{a_\ell} > \frac{a_{s+1}}{a_t} > 
\frac{a_{s+1} a_{t+1}}{a_s a_t}$.  Again $\bar{\bold a}$ induces an
additive cut and is indiscernible, so $\frac{a_{t+1}}{a_t} > 2$ and we
are done.
\end{proof}

\begin{claim}
\label{61}
Let $\bar{\bold a} = \langle a_t:t \in I\rangle$ be some positive
indiscernible sequence endless in both ways in some real closed field
$\cF$.  So for some $\varphi(x,y)$ which defines a function $f$ on
$(\cup\bar{\bold a})^k$ we have that either $f_{I'}(\bar{\bold a})$ or
$f_{I'}(\bar{\bold a}^*)$ induce an additive cut in $\cF$ for some
unbounded $k$-spaced $I' \subseteq I$.
\end{claim}

\begin{proof}
The proof is very similar to the proof of Claim \ref{60}.  The rest of
the proof is the same as the proof there.
\end{proof}

\noindent
\relax From the last two claims we can conclude:
\begin{conclusion}
\label{62}
Let $\bar{\bold a} = \langle a_t:t \in I\rangle$ be some positive
indiscernible sequence endless in both ways.  So for some
$\varphi(x,y)$ which defines a fucntion $f$ on $(\cup \bar{\bold
  a})^k$ we have that either $f_{I'}(\bar{\bold a})$ or
$f_{I'}(\bar{\bold a}^*)$ induce a multiplicative cut in $\cF$ for
some unbounded $k$-spaced $I' \subseteq I$.
\end{conclusion}

\noindent
We now turn to the main theorem in this section.
\begin{theorem}
\label{63}
Let $\cF$ be a real closed field.  Then no two indiscernible sequences
in $\cF$ are strongly perpendicular.
\end{theorem}

\begin{proof}
Let $\bar{\bold a}^1,\bar{\bold a}^2$ be two endless indiscernible
sequences in $\cF$.  We will show that $\bar{\bold a}^1,\bar{\bold
  a}^2$ are not strongly perpendicular.  By Claim \ref{59}, Clause 2
we can assume \wilog \, that both sequences are positive (otherwise
consider the function $f(x) = -x$).  We may also assume \wilog \, that
$\bar{\bold a}^1,\bar{\bold a}^2$ are endless in both ways and
increasing (since every sequence is nb's of any sequence extending it,
and of its reverse).  Now using Claim \ref{62} we apply a definable
function $f$ on the sequences such that the sequences $f(\bar{\bold
  a}^1),f(\bar{\bold a}^2)$ induce multiplicative cuts.  By the
example above (example \ref{48}) $f(\bar{\bold a}^1),f(\bar{\bold
  a}^2)$ has nb's $\bar{\bold b}^1,\bar{\bold b}^2$ respectively such
that $\bar{\bold b}^1,\bar{\bold b}^2$ are not perpendicular. By Claim
\ref{59}, clause 1 above we have that for some $\bar{\bold
  a}^1,\bar{\bold a}^2$ nb's of $\bar{\bold a}^1,\bar{\bold a}^2$
respectively we have that $\bar{\bold b}^i = f(\bar{\bold a}^{i'})$.
By Claim \ref{59}, clause 2 above we have that $\bar{\bold a}^{1'},
\bar{\bold b}^2$ are not perpendicular and then by the same claim
$\bar{\bold a}^{1'},\bar{\bold a}^{2'}$ are not perpendicular. This
completes the proof. 
\end{proof}
\bigskip


\end{document}